\documentclass[12pt, reqno]{amsart}
\usepackage{amssymb,amsthm,amsmath}
\usepackage[numbers,sort&compress]{natbib}
\usepackage{amssymb,amsmath}
\usepackage{amsfonts}
\usepackage{mathrsfs}
\usepackage{latexsym}
\usepackage{amssymb}
\usepackage{amsthm}
\usepackage{indentfirst}
\newtheorem{theorem}{Theorem}[section]
\newtheorem{lemma}{Lemma}[section]
\newtheorem{proposition}{Proposition}[section]
\newtheorem{definition}{Definition}[section]

\begin{document}

\title[Liquid Crystal Equations with Infinite Energy]{Liquid Crystal Equations with Infinite Energy: Local Well-posedness and Blow Up Criterion}
\author{Jinkai Li}
\address[Jinkai Li]{Department of Computer Science and Applied Mathematics, Weizmann Institute of Science, Rehovot, Israel}
\email{jklimath@gmail.com}


\allowdisplaybreaks
\begin{abstract}
In this paper, we consider the Cauchy problem of the incompressible liquid crystal equations in $n$ dimensions. We prove the local well-posedness of mild solutions to the liquid crystal equations with $L^\infty$ initial data, in particular, the initial energy may be infinite. We prove that the solutions are smooth with respect to the space variables away from the initial time. Based on this regularity estimate, we employ the blow up argument and Liouville type theorems to establish vorticity direction type blow up criterions for the type I mild solutions established in the present paper.
\end{abstract}

\maketitle

\section{Introduction}\label{sec1}

In this paper, we consider the following incompressible liquid crystal equations
\begin{align}
&u_t+(u\cdot\nabla)u-\Delta u+\nabla
p=-\textmd{div}(\nabla d\odot\nabla d),\label{1.1}\\
&\textmd{div}u=0,\label{1.2}\\
&d_t+(u\cdot\nabla )d=\Delta d+|\nabla d|^2d \label{1.3}
\end{align}
in $\mathbb R^n\times(0,T)$, $n\geq 2$. Here
$u$ represents the
velocity field of the flow, $d$ represents the
macroscopic molecular orientation of the liquid crystal material,
$p$ denotes the pressure. The symbol $\nabla d\odot\nabla d$, which exhibits the property of
the anisotropy of the material, denotes the $n\times n$ matrix whose
$(i,j)$-th entry is given by $\partial_i d\cdot\partial_jd$ for
$1\leq i,j\leq n$.

System (\ref{1.1})--(\ref{1.3}) is a simplified version of the
Ericksen-Leslie model for the hydrodynamics of nematic liquid crystals
developed by Ericksen \cite{E1}, \cite{E2} and Leslie \cite{LE} in
the 1960's. A brief account of the
Ericksen-Leslie theory and the derivations of several approximate
systems can be found in the appendix of \cite{LL1}, see the two books of Gennes-Prost
\cite{GP} and Chandrasekhar \cite{CH} for more details
of physics. The mathematical analysis of liquid crystal equations is initiated by Lin-Lin in \cite{LL1,LL2} in
the 1990's. They proved in \cite{LL1} the global existence of weak and classical solutions in dimension two or three (for classical solutions in three dimensions, the viscosity coefficient is required to be large) to the Leslie
system of variable length, that is replacing $|\nabla d|^2d$ by the
Ginzburg-Landau type approximation term
$\frac{1-|d|^2}{\varepsilon^2}d$ to relax the nonlinear constraint
$|d|=1$. They proved in \cite{LL2} the
partial regularity theorem for suitable weak solutions, similar to
the classical theorem by Caffarelli-Kohn-Nirenberg \cite{CKN} for
the Navier-Stokes equation. As pointed out in
\cite{LL1,LL2}, both the estimates and arguments in these two papers
depend on $\varepsilon$, and it's a challenging problem to study the
convergence as $\varepsilon$ tends to zero. The two dimensional case
is comparatively easier than the three dimensional one, in fact, Hong \cite{Hong} obtains the convergence as $\varepsilon$ goes to zero up to the first singular time, and thus obtains the existence of weak solutions to the system (\ref{1.1})--(\ref{1.3}). One can also
establish the existence of global weak solutions directly to the
system (\ref{1.1})--(\ref{1.3}) but take the limit for the Ginzburg-Landau
approximate system. Recently, Lin-Lin-Wang \cite{LLW}
proved the global existence of weak solution to the system
(\ref{1.1})--(\ref{1.3}) in dimension two, and obtained the
regularity and asymptotic behavior of the weak solutions they
established. The uniqueness of such weak solution was later proven
in \cite{LW}.

Note that all the papers mentioned in the above consider the solutions with finite energy, i.e. the integration of some quantities of the solutions is finite. A question is what can we say above the weak solution with infinite energy? For the Navier-Stokes equations, the local existence and uniqueness of solutions with bounded initial data and thus allowed to have infinite energy was established by Giga etc in \cite{Giga1}. Such solutions are proven to be global in time in two dimensions later by Giga, Matsui and Sawada in \cite{Giga2}. The Navier-Stokes equations with non decaying initial data are also studied by many authors, see e.g. \cite{Giga3,Sawada,Solonikov} and the references therein. In a recent paper, Giga \cite{Giga4} gives a geometric blow-up criterion on the direction of the vorticity for the three dimensional Navier-Stokes flow whose initial data is just bounded and may have infinite energy. More precisely, by using blow up argument and employing the Liouville type results for two dimensional Navier-Stokes equations, he prove that the bounded type I mild solution (see Definition \ref{def1.2} in the below for the definition of type I) does not blowup if the vorticity direction is uniformly continuous at the place where the vorticity magnitude is large. Such result improves the regularity condition for
the vorticity direction first introduced by Constantin and Fefferman \cite{Constantin} for finite energy weak solution with the payment that the singularity is required to be of type I. For more results on blow up criterion of the vorticity direction, see \cite{Veiga1,Veiga2,Chae1,Chae2,Gruji1,Gruji2,Gruji3,Zhou}.

The goal of this paper is to extend Giga's results \cite{Giga1} on the well-posedness of the Navier-Stokes equations to the liquid crystal equations and establish some blow up criterion on the vorticity direction similar to \cite{Giga4}. More precisely, we establish the local well-posedness of mild solutions to the Cauchy problem of the system (\ref{1.1})--(\ref{1.3}) with bounded initial data which may have infinite energy, and given the regularity criterion on the vorticity direction for type I mild solution. Since the liquid crystals equations is a couple system of the Navier-Stokes equations and the harmonic heat flow equations, besides the vorticity direction assumption, some assumption on the direction field $d$ is required for establishing the blow up criterion on the local solution. In spirit of \cite{Giga4}, we impose some continuity assumption on the direction field. We also prove the regularity of the solutions established in the present paper, in particular, we prove that the velocity and the direction field are smooth with respect to the space variable away from the initial time.

To prove the local well-posedness of mild solutions, we use the Banach fixed point argument, and the key tool is the estimate $\|\nabla e^{t\Delta}\mathbb{P}f\|_\infty\leq Ct^{-1/2}\|f\|_\infty$ proven in \cite{Giga1}. The regularity of the mild solutions is more complicated than the local well-posedness. Since we do not known in advance whether the solutions obtained in the Banach fixed procedure gain higher regularities or not if the initial data is not regular, we need use the approximate argument to do the regularity of the solutions. Our approach can be stated as follows: we first approximate the initial data by a sequence of smooth initial data and thus obtain a sequence of smooth approximate solutions, next, we do the a priori estimates on the higher derivatives of such regularity solutions in terms of the lower ones, next, we prove that the approximate solutions has a common existence time depending only the $L^\infty$ norm of the initial data, and finally, we can take the limit to prove the regularity. On the blow up criterion, thanks to the regularity estimates stated in the above, we use the blow up argument near the singular points to obtain a sequence of bounded mild solutions and take the limit to obtain a limiting vector field $\bar u$ and $\bar d$; by the assumption imposed on the direction field $d$ and the Liouville theorem on harmonic functions, we can prove that $\bar d$ is a constant vector, which implies that $\bar u$ is a bounded backward mild solution to the Navier-Stokes equations, and the rest proof is exactly same to that in Giga \cite{Giga4}.

We complement the equations (\ref{1.1})--(\ref{1.3}) with the following initial data
\begin{equation}
(u, d)|_{t=0}=(u_0, d_0),\label{1.4}
\end{equation}
where $u_0$ and $d_0$ are given vector fields.

Before sating the local well-posedness results, we give the definition of mild solution to the system (\ref{1.1})--(\ref{1.4}).

\begin{definition}\label{def1.1}
We say $(u, d)$ is a mild solution to (\ref{1.1})--(\ref{1.4}) in $\mathbb R^n\times [0, T)$ if it satisfies
\begin{eqnarray*}
  u(t) &=& e^{t\Delta}u_0-\int_0^te^{(t-s)\Delta}\mathbb P \textmd{div}(u\otimes u+\nabla d\odot\nabla d)ds\\
  d(t) &=& e^{t\Delta}d_0+\int_0^te^{(t-s)\Delta}(|\nabla d|^2d-(u\cdot\nabla )d)ds
\end{eqnarray*}
for any $0\leq t<T$, where $\mathbb P$ is the Helmotz projection.
\end{definition}

Our local well-posedness of mild solutions is stated in the following theorem.

\allowdisplaybreaks\begin{theorem}\label{thm1.1}
Let $u_0\in L^\infty$, $d_0\in W^{1,\infty}$ with $\textmd{div}u_0=0$ and $|d_0|=1$. Then there is a unique mild solution $(u,d)$ to system (\ref{1.1})--(\ref{1.4}) in $\mathbb R^n\times[0, T]$ such that $|d|=1$ and

(i) the solution $(u, d)$ has the regularity
$$
u\in C^{1/2}([\delta, T]; W^{k,\infty}) \qquad\mbox{and}\qquad d\in Lip([\delta, T]; W^{k+1,\infty})
$$
for any nonnegative integer $k$ and for any $0<\delta<T$,

(ii) functions $\|u(t)\|_\infty$ and $\|\nabla d(t)\|_\infty$ are both continuous on $[0, T]$, and
$$
\lim_{t\rightarrow 0^+}\|u(t)\|_\infty=\|u_0\|_\infty\qquad \mbox{and}\qquad \lim_{t\rightarrow0^+}\|\nabla d(t)\|_\infty=\|\nabla d_0\|_\infty,
$$

(iii) the existence time $T$ satisfies
$$
T\geq C(\|u_0\|_\infty+\|\nabla d_0\|_\infty)^{-2}
$$
for some positive constant $C$ depending only on $n$.
\end{theorem}

Now, we give the definition of type I mild solution.

\begin{definition}\label{def1.2}
A mild solution $(u, d)$ defined on $\mathbb R^n\times[-1, 0)$ to the system (\ref{1.1})--(\ref{1.3}) is said to be a type I mild solution if
$$
\|u(t)\|_\infty+\|\nabla d(t)\|_\infty\leq C(-t)^{-1/2},\qquad t\in(-1, 0)
$$
for some positive constant $C$.
\end{definition}

Our blow up criterion are stated in the following two theorems.

\begin{theorem}\label{thm1.2}
Let $(u, d)$ be a type I mild solution to system (\ref{1.1})--(\ref{1.3}) in $\mathbb R^3\times(-1, 0)$. For given $\sigma>0$, let $\eta$ be a modulus, such that
\begin{eqnarray*}
  |\zeta(x,t)-\zeta(y, t)|&\leq& \eta(|x-y|) ,\qquad\forall x, y\in\Omega_\sigma(t),\\
  |d(x, t)-d(y,t)| &\leq&\eta(|x-y|),\qquad\forall x, y\in\mathbb R^3,
\end{eqnarray*}
where $\Omega_\sigma(t)=\big\{x\in\mathbb R^3\big||\omega(x,t)|>\sigma\big\}$, $\zeta=\frac{\omega}{|\omega|}$ and $\omega=\textmd{curl}u$. Then $(u, d)$ does not blow up at $t=0$, that is, we can extend $(u, d)$ to be a mild solution to system (\ref{1.1})--(\ref{1.3}) in $\mathbb R^3\times(-1, \varepsilon)$ for some $\varepsilon>0$.
\end{theorem}

\begin{theorem}\label{thm1.3}
Let $(u, d)$ be a type I mild solution to system (\ref{1.1})--(\ref{1.3}) in $\mathbb R^3\times(-1, 0)$. For given $\sigma>0$, let $\eta$ be a modulus, such that
\begin{align*}
  &\int_{-1}^0\|\nabla\zeta\|_{L^b(\Omega_\sigma(t))}^adt<\infty,\quad\frac{2}{a}+\frac{3}{b}\leq1, ~~2\leq a<\infty,\\
  &|d(x, t)-d(y,t)| \leq\eta(|x-y|),\qquad\forall x, y\in\mathbb R^3,
\end{align*}
where $\Omega_\sigma(t)=\big\{x\in\mathbb R^3\big||\omega(x,t)|>\sigma\big\}$, $\zeta=\frac{\omega}{|\omega|}$ and $\omega=\textmd{curl}u$. Then $(u, d)$ does not blow up at $t=0$, that is, we can extend $(u, d)$ to be a mild solution to system (\ref{1.1})--(\ref{1.3}) in $\mathbb R^3\times(-1, \varepsilon)$ for some $\varepsilon>0$.
\end{theorem}

The rest of this paper is arranged as follows: in the next section, Section \ref{sec2}, we prove the local existence of mild solutions with smooth initial data by using the Banach fixed point argument; in Section \ref{sec3}, we do the a priori estimates on the higher derivatives in terms of the lower ones of the mild solutions; Section \ref{sec4} is employed to estimate the existence time of the smooth mild solutions in terms of the $L^\infty$ norm of the initial data; in section \ref{sec5}, using the results in the previous sections, we prove the local well-posedness of mild solutions with $L^\infty$ initial data and the regularities of the solutions; the proof of the blow up criterion is given in the last section.

\section{Local existence with smooth initial data}\label{sec2}

In this section, we prove the local existence of mild solutions with smooth initial data, that is the following proposition.

\begin{proposition}\label{prop2.1}
Let $k$ be a nonnegative integer, $u_0\in W^{k,\infty}$, $\textmd{div}u_0=0$, $d_0\in W^{k+1, \infty}$ and $|d_0|\leq 1$. Then there is a positive $T$, such that the system (\ref{1.1})--(\ref{1.4}) has a mild solution $(u, d)$ on $\mathbb R^n\times(0, T)$ satisfying $(u, d)\in L^\infty(0, T; W^{k,\infty}\times W^{k+1, \infty})$, $|d|\leq 1$ and
$$
\lim_{t\rightarrow 0^+}\|u(t)\|_\infty=\|u_0\|_\infty\quad\mbox{and}\quad\lim_{t\rightarrow 0^+}\|\nabla d(t)\|_\infty=\|\nabla d_0\|_\infty.
$$
\end{proposition}

We are going to use Banach fixed point argument to prove the above proposition. As a preparation, the following lemma plays a key role in our argument.

\begin{lemma}\label{lem2.1}
Let $n\geq 1$ and $1\leq p\leq q\leq\infty$. Then it follows
\begin{eqnarray*}
  \|\partial_x^\beta e^{t\Delta}f\|_q &\leq& Ct^{-\frac{n}{2}\left(\frac{1}{p}-\frac{1}{q}\right)-\frac{|\beta|}{2}}\|f\|_p, \\
  \|e^{t\Delta}\mathbb P\partial_i f\|_p &\leq& Ct^{-\frac{1}{2}}\|f\|_p
\end{eqnarray*}
for any $t>0$, $\beta\in \mathbb N_0^n$ and $f\in L^p$, where $C$ is a positive constant depending only on $n$.
\end{lemma}

\begin{proof}
The first one well known, while the second one can be found in \cite{Giga1}.
\end{proof}

Now we define the Banach spaces we will work in and the map being considered. Given $T>0$ and nonnegative integer $k$, denote by $X_T^k$ the space
$$
X_T^k=\{(u, d) | (u, d)\in L^\infty(0, T; W^{k,\infty}\times W^{k+1, \infty})\}
$$
with norm
$$
\|(u, d)\|_{X_T^k}=\|u\|_{L^\infty(0, T; W^{k,\infty})}+\|d\|_{L^\infty(0, T; W^{k+1, \infty})}.
$$
Then $X_T^k$ is a Banach space. Denote by $\mathcal B_R$ be the ball in $X_T^k$ with radius $R$ and center zero. Given $(u_0, d_0)\in W^{k,\infty}\times W^{k+1,\infty}$ with $\textmd{div}u_0=0$ and $|d_0|\leq1$. Define operator $\mathcal S$ on $X_T^k$ as $\mathcal S(u, d) = (\tilde u, \tilde d) $ with $\tilde u$ and $\tilde d$ given by
\begin{eqnarray*}
  \tilde u(t) &=& e^{t\Delta}u_0-\int_0^te^{(t-s)\Delta }\mathbb P\textmd{div}(u\otimes u+\nabla d\odot\nabla d)ds \\
  \tilde d(t) &=& e^{t\Delta}d_0+\int_0^te^{(t-s)\Delta}(|\nabla d|^2d-(u\cdot\nabla)d)ds
\end{eqnarray*}
for any $0\leq t\leq T$.

\begin{lemma}\label{lem2.2}
Let $\mathcal S$ be the map defined in the above. Then $\mathcal S$ maps $X_T^k$ into $X_T^k$, and the following estimates hold ture
$$
\|\mathcal S(u, d)\|_{X_T^k}\leq C(\|u_0\|_{W^{k,\infty}}+\|d_0\|_{W^{k+1,\infty}})+C(T^{1/2}+T)(\|(u, d)\|_{X_T^k}^2+\|(u, d)\|_{X_T^k}^3)
$$
for any $(u, d)\in X_T^k$, where $C$ is a positive constant depending only on $n$ and $k$.
\end{lemma}

\begin{proof}
Take arbitrary $(u, d)\in X_T^k$. By Lemma \ref{lem2.1}, it follows
\begin{align*}
\|\nabla^l\tilde u\|_\infty(t)\leq&\|\nabla^le^{t\Delta}u_0\|_\infty+\int_0^t\|e^{(t-s)\Delta}\mathbb P\textmd{div}\nabla^l(u\otimes u+\nabla d\odot\nabla d)\|_\infty ds\\
\leq&C\|\nabla^lu_0\|_\infty+C\int_0^t(t-s)^{-1/2}\sum_{i=0}^l(\|\nabla^iu\|_\infty\|\nabla^{l-i}u\|_\infty\\
&+\|\nabla^{i+1} d\|_\infty\|\nabla^{l-i+1}d\|_\infty)ds\\
\leq&C\|\nabla^lu_0\|_\infty+C\int_0^t(t-s)^{-1/2}(\|u\|_{W^{k,\infty}}^2+\|d\|_{W^{k+1,\infty}}^2)ds
\end{align*}
for any $0\leq l\leq k$, and
\begin{align*}
\|\nabla^m\tilde d\|_\infty(t)\leq&\|\nabla^me^{t\Delta}d_0\|_\infty+\int_0^t\|\nabla^me^{(t-s)\Delta}(|\nabla d|^2d-(u\cdot\nabla )d)\|_\infty ds\\
\leq&C\|\nabla^md_0\|_\infty+C\int_0^t(t-s)^{-1/2}\|\nabla^{m-1}(|\nabla d|^2d-(u\cdot\nabla)d\|_\infty)ds\\
\leq&C\|\nabla^m d_0\|_\infty+C\int_0^t(t-s)^{-1/2}\sum_{{0\leq i,j\leq m-1}\atop{i+j\leq m-1}}(\|\nabla^iu\|_\infty\|\nabla^{m-i}d\|_\infty\\
&+\|\nabla^{i+1}d\|_\infty\|\nabla^{j+1}d\|_\infty\|\nabla^{m-1-i-j}d\|_\infty)ds\\
\leq&C\|\nabla^md_0\|_\infty+C\int_0^t(t-s)^{-1/2}(\|d\|_{W^{k+1,\infty}}^3+\|u\|_{W^{k,\infty}}\|d\|_{W^{k+1,\infty}})ds
\end{align*}
for any $1\leq m\leq k+1$ and
$$
\|\tilde d\|_\infty(t)\leq C\|d_0\|_\infty+C\int_0^t(\|\nabla d\|_\infty^2\|d\|_\infty+\|u\|_\infty\|\nabla d\|_\infty)ds,
$$
and thus
\begin{align*}
&\|\tilde u(t)\|_{W^{k,\infty}}+\|\tilde d(t)\|_{W^{k+1,\infty}}\\
\leq&C(\|u_0\|_{W^{k,\infty}}+\|d_0\|_{W^{k+1,\infty}})+C\int_0^t(t-s)^{-1/2}(\|u\|_{W^{k,\infty}}^2+\|d\|_{W^{k+1,\infty}}^2+\|d\|_{W^{k+1,\infty}}^3)ds\\
&+C\int_0^t(\|\nabla d\|_\infty^2\|d\|_\infty+\|u\|_\infty\|\nabla d\|_\infty)ds\\
\leq&C(\|u_0\|_{W^{k,\infty}}+\|d_0\|_{W^{k+1,\infty}})+C(T^{1/2}+T)(\|(u, d)\|_{X_T^k}^2+\|(u, d)\|_{X_T^k}^3),
\end{align*}
which gives
$$
\|\mathcal S(u,d)\|_{X_T^k}\leq C(\|u_0\|_{W^{k,\infty}}+\|d_0\|_{W^{k+1,\infty}})+C(T^{1/2}+T)(\|(u, d)\|_{X_T^k}^2+\|(u, d)\|_{X_T^k}^3),
$$
completing the proof.
\end{proof}

\begin{lemma}\label{lem2.3}
The following estimates hold true
$$
\|\mathcal S(u, d)-\mathcal S(v,n)\|_{X_T^k}\leq C(K+K^2)(T^{1/2}+T)\|(u,d)-(v,n)\|_{X_T^k}
$$
for any $(u,d), (v,n)\in \mathcal B_K$, where $C$ is a positive constant depending only on $n$ and $k$.
\end{lemma}

\begin{proof}
Take arbitrary $(u, d), (v,n)\in \mathcal B_K$. By definition $\mathcal S(u,d)=(\tilde u,\tilde d)$ and $\mathcal S(v,n)=(\tilde v,\tilde n)$ and
\begin{align*}
\tilde u(t)-\tilde v(t)=&\int_0^t e^{(t-s)\Delta}\mathbb P\textmd{div}(v\otimes v-u\otimes u+\nabla n\odot\nabla n-\nabla d\odot\nabla d)ds\\
=&\int_0^te^{(t-s)\Delta}\mathbb P\textmd{div}((v-u)\otimes v+u\otimes(v-u)\\
&+(\nabla n-\nabla d)\odot\nabla n+\nabla d\odot(\nabla n-\nabla d))ds,
\end{align*}
and
\begin{align*}
\tilde d(t)-\tilde n(t)=&\int_0^te^{(t-s)\Delta}(|\nabla d|^2d-|\nabla n|^2n+(v\cdot\nabla)n-(u\cdot\nabla)d)ds\\
=&\int_0^te^{(t-s)\Delta}((\nabla d-\nabla n)(\nabla d+\nabla n)d+|\nabla n|^2(d-n)\\
&+(v-u)\nabla n+u\nabla(n-d))ds.
\end{align*}
Hence, by Lemma \ref{lem2.1}, it follows
\begin{align*}
\|\nabla^l(\tilde u-\tilde v)\|_{\infty}(t)\leq&\int_0^t\|e^{(t-s)\Delta}\mathbb P\textmd{div}\nabla^l((v-u)\otimes v+u\otimes(v-u)\\
&+(\nabla n-\nabla d)\odot\nabla n+\nabla d\odot(\nabla n-\nabla d))ds\\
\leq&C\int_0^t(t-s)^{-1/2}\sum_{i=0}^l[\|\nabla^i(v-u)\|_\infty(\|\nabla^{l-i}u\|_\infty+\|\nabla^{l-i}v\|_\infty)\\
&+\|\nabla^{i+1}(n-d)\|_\infty(\|\nabla^{l-i+1}n\|_\infty+\|\nabla^{l-i+1}d\|_\infty)]ds\\
\leq&C\int_0^t(t-s)^{-1/2}(\|u\|_{W^{k,\infty}}+\|v\|_{W^{k,\infty}}+\|d\|_{W^{k+1,\infty}}+\|n\|_{W^{k+1,\infty}})\\
&\times(\|u-v\|_{W^{k,\infty}}+\|d-n\|_{W^{k+1,\infty}})ds\\
\leq&CK\|(u,d)-(v,n)\|_{X_T^k}\int_0^t(t-s)^{-1/2}ds\\
=&CT^{1/2}K\|(u,d)-(v,n)\|_{X_T^k}
\end{align*}
for any $0\leq l\leq k$, and
\begin{align*}
&\|\nabla^m(\tilde d-\tilde n)\|_\infty(t)\\
\leq&\int_0^t\|\nabla e^{(t-s)\Delta}\nabla^{m-1}((\nabla d-\nabla n)(\nabla d+\nabla n)d+|\nabla n|^2(d-n)\\
&+(v-u)\nabla n+u\nabla(n-d))\|_\infty ds\\
\leq&C\int_0^t(t-s)^{-1/2}\sum_{{0\leq i,j\leq m-1}\atop{i+j\leq m-1}}(\|\nabla^{i+1}(d-n)\|_\infty\|\nabla^{j+1}(d+n)\|_\infty\|\nabla^{m-1-i-j}d\|_\infty\\
&+\|\nabla^{i+1}n\|_\infty\|\nabla^{j+1}n\|_\infty\|\nabla^{m-1-i-j}(d-n)\|_\infty+\|\nabla^i(v-u)\|_\infty\|\nabla^{m-i}n\|_\infty\\
&+\|\nabla^iu\|_\infty\|\nabla^{m-i}(n-d)\|_\infty)ds\\
\leq&C\int_0^t(t-s)^{-1/2}(\|d\|_{W^{k+1,\infty}}^2+\|n\|_{W^{k+1,\infty}}^2+\|n\|_{W^{k+1,\infty}}+\|u\|_{W^{k,\infty}})\\
&\times(\|d-n\|_{W^{k+1,\infty}}+\|u-v\|_{W^{k,\infty}})ds\\
\leq&C(K+K^2)\|(u,d)-(v,n)\|_{X_T^k}\int_0^t(t-s)^{-1/2}ds\\
\leq&C(K+K^2)T^{1/2}\|(u,d)-(v,n)\|_{X_T^k}
\end{align*}
for any $1\leq m\leq k+1$, and
\begin{align*}
\|\tilde d-\tilde n\|_\infty(t)
\leq&\int_0^t\| e^{(t-s)\Delta}((\nabla d-\nabla n)(\nabla d+\nabla n)d+|\nabla n|^2(d-n)\\
&+(v-u)\nabla n+u\nabla(n-d))\|_\infty ds\\
\leq&C\int_0^t(\|\nabla(d-n)\|_\infty\|\nabla(d+n)\|_\infty\|d\|_\infty+\|\nabla n\|_\infty^2\|d-n\|_\infty\\
&+\|v-u\|_\infty\|\nabla n\|_\infty+\|u\|_\infty\|\nabla(n-d)\|_\infty)ds\\
\leq&C\int_0^t(\|d\|_{W^{1,\infty}}^2+\|n\|_{W^{1,\infty}}^2+\|n\|_{W^{1,\infty}}+\|u\|_\infty)\\
&\times(\|u-v\|_\infty+\|n-d\|_\infty)ds\\
\leq&C(K+K^2)T\|(u,d)-(v,n)\|_{X_T^k}.
\end{align*}
Hence, we have
$$
\|(\tilde u, \tilde d)-(\tilde v, \tilde n)\|_{X_T^k}\leq C(K+K^2)(T^{1/2}+T)\|(u, d)-(v, n)\|_{X_T^k},
$$
proving the conclusion.
\end{proof}

\begin{lemma}\label{lem2.4}
$\mathcal S$ maps $\mathcal B_{K_*}$ into itself, and
$$
\|\mathcal S(u, d)-\mathcal S(v, n)\|_{X_{T_*}^k}\leq\frac{1}{2}\|(u, d)-(v, n)\|_{X_{T_*}^k}
$$
for any $(u, d), (v, n)\in\mathcal B_{K_*}$, where $T_*$ and $K_*$ are given by
$$
K_*=2C_*(\|u_0\|_{W^{k,\infty}}+\|d_0\|_{W^{k+1,\infty}})
$$
and
$$
T_*=\min\left\{\frac{1}{4C_*(K_*+K_*^2)}, \frac{1}{16C_*^2(K_*+K_*^2)^2}\right\}
$$
for some positive constant $C_*$ depending only on $n$ and $k$.
\end{lemma}

\begin{proof}
By Lemma \ref{lem2.2} and Lemma \ref{lem2.3}, there is a positive constant $C_*$ depending only on $n$ and $k$, such that
$$
\|\mathcal S(u, d)\|_{X_T^k}\leq C_*(\|u_0\|_{W^{k,\infty}}+\|d_0\|_{W^{k+1,\infty}})+C_*(T^{1/2}+T)(K^2+K^3),
$$
and
$$
\|\mathcal S(u, d)-\mathcal S(v, n)\|_{X_T^k}\leq C_*(T^{1/2}+T)(K+K^2)\|(u, d)-(v, n)\|_{X_T^k}
$$
for any $(u, d), (v, n)\in\mathcal B_K$, where $C_*$ is a positive constant depending only on $n$ and $k$. Set
$$
K_*=2C_*(\|u_0\|_\infty+\|d_0\|_{W^{k+1,\infty}})
$$
and
$$
T_*=\min\left\{\frac{1}{4C_*(K_*+K_*^2)}, \frac{1}{16C_*^2(K_*+K_*^2)^2}\right\}.
$$
Then one can easily check that
$$
C_*(\|u_0\|_{W^{k,\infty}}+\|d_0\|_{W^{k+1,\infty}})\leq\frac{K_*}{2}\quad\mbox{and}\quad C_*(T_*^{1/2}+T_*)(K_*+K_*^2)\leq\frac{1}{2},
$$
and thus
$$
\|\mathcal S(u, d)\|_{X_{T_*}^k}\leq C_*(\|u_0\|_{W^{k,\infty}}+\|d_0\|_{W^{k+1,\infty}})+C_*(T_*^{1/2}+T_*)(K_*^2+K_*^3)\leq\frac{K_*}{2}+\frac{K_*}{2}\leq K_*,
$$
and
\begin{align*}
\|\mathcal S(u, d)-\mathcal S(v, n)\|_{X_{T_*}^k}\leq& C_*(T_*^{1/2}+T_*)(K_*+K_*^2)\|(u, d)-(v, n)\|_{X_{T_*}^k}\\
\leq&\frac{1}{2}\|(u, d)-(v, n)\|_{X_{T_*}^k}
\end{align*}
for any $(u, d), (v, n)\in\mathcal B_{K_*}$. The proof is complete.
\end{proof}

Now, we can give the proof of Proposition \ref{prop2.1} as follows.

\textbf{Proof of Proposition \ref{prop2.1}} Let $T_*$ and $K_*$ be the constant stated in Lemma \ref{lem2.4}. By Lemma \ref{lem2.4}, $\mathcal S$ is a contradictive map from $\mathcal B_{K_*}$ to itself. By Banach's fixed point theorem, there is a unique fixed point in $\mathcal B_{K_*}$ to $\mathcal S(u, d)$, that is, there is a unique solution $(u, d)$ in $\mathcal B_{K_*}$ to equation
$$
\mathcal S(u, d)=(u, d).
$$
By definition of operator $\mathcal S$, such $(u, d)$ is a mild solution to the system (\ref{1.1})--(\ref{1.4}) in $\mathbb R^n\times [0, T_*]$. Note that
$$
d(t)=e^{t\Delta}d_0+\int_0^te^{(t-s)\Delta}(|\nabla d|^2d-(u\cdot\nabla)d)ds.
$$
Calculate directly, there holds
\begin{align*}
\partial_td=&\Delta e^{t\Delta }d_0+\int_0^t\Delta e^{(t-s)\Delta}(|\nabla d|^2d-(u\cdot\nabla)d)ds+|\nabla d|^2d-(u\cdot\nabla)d\\
=&\Delta\left(e^{t\Delta}d_0+\int_0^te^{(t-s)\Delta}(|\nabla d|^2d-(u\cdot\nabla)d)ds\right)+|\nabla d|^2d-(u\cdot\nabla)d\\
=&\Delta d+|\nabla d|^2d-(u\cdot\nabla)d,
\end{align*}
or equivalently
$$
d_t-\Delta d-|\nabla d|^2d+(u\cdot\nabla)d=0.
$$
As a consequence, the scalar function $l(x,t)=|d|^2(x,t)-1$ satisfies
$$
\left\{
\begin{array}{l}
\partial_t l-\Delta l-2|\nabla d|^2l+u\nabla l=0,\\
l(x,0)=|d_0(x)|^2-1\leq 0.
\end{array}
\right.
$$
Maximum principle of parabolic equations implies that $l\leq0$, which gives $|d|\leq 1$.

Now we show that
$$
\lim_{t\rightarrow 0^+}(\|u(t)\|_\infty+\|\nabla d(t)\|_\infty)=\|u_0\|_\infty+\|\nabla d_0\|_\infty.
$$
Since $(u, d)$ is a mild solution, we have
\begin{eqnarray*}
  u(t) &=& e^{t\Delta }u_0-\int_0^te^{(t-s)\Delta}\mathbb P\textmd{div}(u\otimes u+\nabla d\odot\nabla d)ds, \\
  d(t) &=& e^{t\Delta }d_0+\int_0^te^{(t-s)\Delta}(|\nabla d|^2d-(u\cdot\nabla)d)ds.
\end{eqnarray*}
Hence, it follows from Lemma \ref{lem2.1} that
\begin{align*}
\|u(t)\|_\infty\leq&\|e^{t\Delta }u_0\|_\infty+\int_0^t\|e^{(t-s)\Delta}\mathbb P\textmd{div}(u\otimes u+\nabla d\odot\nabla d)\|_\infty ds\\
\leq&\|u_0\|_\infty+C\int_0^t(t-s)^{-1/2}(\|u\|_\infty^2+\|\nabla d\|_\infty^2)ds\leq\|u_0\|_\infty+Ct^{1/2},
\end{align*}
and
\begin{align*}
\|\nabla d(t)\|_\infty\leq&\|e^{t\Delta }\nabla d_0\|_\infty+\int_0^t\|\nabla e^{(t-s)\Delta}(|\nabla d|^2d-(u\cdot\nabla)d)ds\\
\leq&\|\nabla d_0\|_\infty+C\int_0^t(t-s)^{-1/2}(\|\nabla d\|_\infty^2+\|u\|_\infty\|\nabla d\|_\infty)ds\\
\leq&\|\nabla d_0\|_\infty+Ct^{1/2}
\end{align*}
for any $0<t\leq T_*$. Combining the above inequalities, we conclude
\begin{equation}\label{2.1}
\limsup_{t\rightarrow 0^+}\|u(t)\|_\infty\leq\|u_0\|_\infty\quad\mbox{and}\quad\limsup_{t\rightarrow 0^+}\|\nabla d(t)\|_\infty\leq\|\nabla d_0\|_\infty.
\end{equation}

We claim that $(u_0(t), \nabla d(t))\rightarrow (u_0, \nabla d_0)$ as $t\rightarrow 0^+$ weak star in $L^\infty$. For this aim, take arbitrary $\varphi\in C_0^\infty$, and then
\begin{align}
&\int_{\mathbb R^n}u(x,t)\varphi(x)dx\nonumber\\
=&\int_{\mathbb R^n}e^{t\Delta}u_0\varphi(x)dx-\int_{\mathbb R^n}\left(\int_0^te^{(t-s)\Delta}\mathbb P\textmd{div}(u\otimes u+\nabla d\odot\nabla d) ds\right)\varphi(x)dx\nonumber\\
=&\int_{\mathbb R^n}u_0e^{t\Delta}\varphi(x)dx-\int_{\mathbb R^n}\left(\int_0^te^{(t-s)\Delta}\mathbb P\textmd{div}(u\otimes u+\nabla d\odot\nabla d) ds\right)\varphi(x)dx,\label{2.2}
\end{align}
and
\begin{align}
&\int_{\mathbb R^n}\partial_id(x,t)\varphi(x)dx=-\int_{\mathbb R^n}d(x, t)\partial_i\varphi(x)dx\nonumber\\
=&-\int_{\mathbb R^n}e^{t\Delta}d_0\partial_i\varphi(x)dx-\int_{\mathbb R^n}\left(\int_0^te^{(t-s)\Delta}(|\nabla d|^2d-(u\cdot\nabla)d)ds\right)\partial_i\varphi dx.\label{2.3}
\end{align}
By Lemma \ref{lem2.1}, it follows
\begin{align}
&\left|\int_{\mathbb R^n}\left(\int_0^te^{(t-s)\Delta}\mathbb P\textmd{div}(u\otimes u+\nabla d\odot\nabla d) ds\right)\varphi(x)dx\right|\nonumber\\
\leq&\int_{\mathbb R^n}\int_0^t\|e^{(t-s)\Delta}\mathbb P\textmd{div}(u\otimes u+\nabla d\odot\nabla d)\|_\infty ds|\varphi(x)|dx\nonumber\\
\leq&C\int_{\mathbb R^n}\int_0^t(t-s)^{-1/2}(\|u\|_\infty^2+\|\nabla d\|_\infty^2)ds|\varphi(x)|dx\nonumber\\
\leq&C\int_{\mathbb R^n}t^{1/2}|\varphi(x)|dx\rightarrow 0,\qquad \mbox{as }t\rightarrow 0^+,\label{2.4}
\end{align}
and
\begin{align}
&\left|\int_{\mathbb R^n}\left(\int_0^te^{(t-s)\Delta}(|\nabla d|^2d-(u\cdot\nabla)d)ds\right)\partial_i\varphi(x)dx\right|\nonumber\\
\leq&\int_{\mathbb R^n}\int_0^t\|e^{(t-s)\Delta}(|\nabla d|^2d-(u\cdot\nabla)d)\|_\infty ds|\partial_i\varphi(x)|dx\nonumber\\
\leq&C\int_{\mathbb R^n}\int_0^t(\|\nabla d\|_\infty^2+\|u\|_\infty\|\nabla d\|_\infty)ds|\partial_i\varphi(x)|dx\nonumber\\
\leq&C\int_{\mathbb R^n}|\partial_i\varphi(x)|tdx\rightarrow 0,\qquad \mbox{as }t\rightarrow 0^+.\label{2.5}
\end{align}
Using Lemma \ref{lem2.1} again, it has
\begin{align*}
&\|e^{t\Delta}\varphi-\varphi\|_1=\left\|\int_0^t e^{s\Delta}\Delta \varphi ds\right\|_1\\
\leq&\int_0^t\|e^{s\Delta}\Delta \varphi\|_1 ds\leq\int_0^t\|\Delta\varphi\|_1 ds\\
=&\|\Delta\varphi\|_1 t\rightarrow 0,\qquad\mbox{ as }t\rightarrow 0^+,
\end{align*}
and
$$
\|e^{t\Delta}\partial_i\varphi-\partial_i\varphi\|_1\leq\|\partial_i\Delta\varphi\|_1 t\rightarrow 0,\qquad\mbox{as }t\rightarrow0^+.
$$
Consequently, we have
\begin{equation}
  \int_{\mathbb R^n}u_0e^{t\Delta}\varphi dx\rightarrow\int_{\mathbb R^n}u_0\varphi dx,\qquad \mbox{as }t\rightarrow 0^+, \label{2.6}
\end{equation}
and
\begin{equation}
  \int_{\mathbb R^n}e^{t\Delta}d_0\partial_i\varphi dx=\int_{\mathbb R^n}d_0e^{t\Delta}\partial_i\varphi dx\rightarrow\int_{\mathbb R^n}d_0\partial_i\varphi dx,\quad \mbox{as }t\rightarrow0^+.\label{2.7}
\end{equation}
Combining (\ref{2.2})--(\ref{2.7}), it follows
\begin{equation*}
\lim_{t\rightarrow0^+}\int_{\mathbb R^n}u(x,t)\varphi(x)dx=\int_{\mathbb R^n}u_0\varphi(x)dx,
\end{equation*}
and
\begin{equation*}
\lim_{t\rightarrow0^+}\int_{\mathbb R^n}\partial_id(x,t)\varphi(x)dx=\int_{\mathbb R^n}\partial_id_0\varphi(x)dx
\end{equation*}
for any $\varphi\in C_0^\infty$. Recalling that $(u,\nabla d)\in L^\infty(0, T_*;L^\infty)$ and $C_0^\infty$ is dense in $L^1$, we conclude that
$$
u(t)\rightarrow u_0\qquad\mbox{and }\qquad\nabla d(t)\rightarrow\nabla d_0,\quad \mbox{ as }t\rightarrow0^+,
$$
weak star in $L^\infty$. Hence we have
$$
\|u_0\|_\infty\leq\liminf_{t\rightarrow0^+}\|u(t)\|_\infty\qquad\mbox{ and }\qquad\|\nabla d_0\|_\infty\leq\liminf_{t\rightarrow 0^+}\|\nabla d(t)\|_\infty.
$$
Combining this with (\ref{2.1}), we conclude
$$
\lim_{t\rightarrow 0^+}\|u(t)\|_\infty\leq\|u_0\|_\infty\quad\mbox{and}\quad\lim_{t\rightarrow 0^+}\|\nabla d(t)\|_\infty\leq\|\nabla d_0\|_\infty.
$$
The proof is complete.

\section{A priori estimates for regular solutions}\label{sec3}

In this section we study the a priori estimates on higher derivatives of the regular mild solutions in terms of the lower ones, that is the following proposition.

\begin{proposition}\label{prop3.1}
Let $k\geq 2$ and $(u, d)\in X_T^k$ be a mild solution to system (\ref{1.1})--(\ref{1.4}), where $X_T^k$ is the Banach space stated in Section \ref{sec2}. Suppose that
$$
\|u\|_{L^\infty(0,T; L^\infty)}+\|\nabla d\|_{L^\infty(0, T; L^\infty)}\leq M
$$
for some positive constant $M$. Then  we have the following estimates
\begin{eqnarray*}
&\|u(t)-u(t_0)\|_{W^{k-1,\infty}}\leq C(|t-t_0|+|t-t_0|^{\frac{1}{2k}})\\
&\|(u, d)\|_{L^{\infty}(\delta, T; W^{k,\infty}\times W^{k+1,\infty})}\leq C\\
&\|d(t)-d(t_0)\|_{W^{k,\infty}}\leq C(|t-t_0|+|t-t_0|^{\frac{1}{k+1}})
\end{eqnarray*}
for any $\delta\leq t_0\leq t\leq T$ and $0<\delta<T$, where $C$ is a positive constant depending only on $n$, $k$, $\delta$ and $M$.
\end{proposition}

\begin{proof}
Since $(u, d)$ is a mild solution to (\ref{1.1})--(\ref{1.4}) in $\mathbb R^n\times[0, T]$, it follows
\begin{eqnarray*}
  u(t) &=& e^{(t-t_0)\Delta }u(t_0)-\int_{t_0}^te^{(t-s)\Delta}\mathbb P\textmd{div}(u\otimes u+\nabla d\odot\nabla d)ds, \\
  d(t) &=& e^{(t-t_0)\Delta }d(t_0)+\int_{t_0}^te^{(t-s)\Delta}(|\nabla d|^2d-(u\cdot\nabla)d)ds
\end{eqnarray*}
for any $0\leq t_0\leq t\leq T$.

We use mathematical inductive argument on $l$ to prove that
$$
\|(u, \nabla d)\|_{L^\infty(\delta, T; W^{l,\infty})}\leq C(\delta, M, n)\quad 0\leq l\leq k
$$
for any $0<\delta<T$. For $l=0$, the estimates follow trivially from our assumption $\|u\|_{L^\infty(0,T; L^\infty)}+\|\nabla d\|_{L^\infty(0, T; L^\infty)}\leq M$. Suppose that the estimates hold true for any $0\leq l\leq l_0$ and any $0<\delta<T$. Then for any $t_0\geq\frac{\delta}{2}$, we have
\allowdisplaybreaks\begin{align*}
&\|\nabla^{l_0+1}u\|_\infty(t)\\
\leq&\|\nabla e^{(t-t_0)\Delta}\nabla^{l_0}u(t_0)\|_\infty+\int_{t_0}^t\|e^{(t-s)\Delta}\mathbb P\textmd{div}\nabla^{l_0+1}(u\otimes u+\nabla d\odot\nabla d)\|_\infty ds\\
\leq&C\int_{t_0}^t(t-s)^{-1/2}\sum_{i=0}^{l_0+1}(\|\nabla^iu\|_\infty\|\nabla^{l_0+1-i}u\|_\infty
+\|\nabla^{i+1}d\|_\infty\|\nabla^{l_0+2-i}d\|_\infty)ds\\
&+C(t-t_0)^{-1/2}\|\nabla^{l_0}u(t_0)\|_\infty\\
\leq&C\int_{t_0}^t(t-s)^{-1/2}(\|u\|_{W^{l_0,\infty}}^2+\|\nabla d\|_{W^{l_0+1,\infty}}^2+\|u\|_\infty\|\nabla^{l_0+1}u\|_\infty\\
&+\|\nabla d\|_\infty\|\nabla^{l_0+2}d\|_\infty)ds+C(t-t_0)^{-1/2}\|\nabla^{l_0}u(t_0)\|_\infty\\
\leq&C(t-t_0)^{-1/2}+C\int_{t_0}^t(t-s)^{-1/2}(1+\|\nabla^{l_0+1}u\|_\infty+\|\nabla^{l_0+2}d\|_\infty)ds\\
=&C[(t-t_0)^{-1/2}+(t-t_0)^{1/2}]+C\int_{t_0}^t(t-s)^{-1/2}(\|\nabla^{l_0+1}u\|_\infty+\|\nabla^{l_0+2}d\|_\infty)ds,
\end{align*}
and
\begin{align*}
&\|\nabla^{l_0+2}d\|_\infty(t)\\
\leq&\|\nabla e^{(t-t_0)\Delta}\nabla^{l_0+1}d(t_0)\|_\infty+\int_{t_0}^t\|\nabla e^{(t-s)\Delta}\nabla^{l_0+1}(|\nabla d|^2d-(u\cdot\nabla)d)\|_\infty ds\\
\leq&C(t-t_0)^{-1/2}\|\nabla^{l_0+1}d(t_0)\|_\infty+C\int_{t_0}^t(t-s)^{-1/2}\sum_{{0\leq i,j\leq l_0+1}\atop{i+j\leq l_0+1}}(\|\nabla^iu\|_\infty\|\nabla^{l_0+2-i}d\|_\infty\\
&+\|\nabla^{i+1}d\|_\infty\|\nabla^{j+1}d\|_\infty\|\nabla^{l_0+1-i-j}d\|_\infty)ds\\
\leq&C(t-t_0)^{-1/2}+C\int_{t_0}^t(t-s)^{-1/2}(\|d\|_{W^{l_0+1,\infty}}^3+\|\nabla d\|_\infty\|\nabla^{l_0+2}d\|_\infty\\
&+\|u\|_{W^{l_0,\infty}}\|d\|_{W^{l_0+1,\infty}}+\|u\|_\infty\|\nabla^{l_0+2}d\|_\infty+\|\nabla^{l_0+1}u\|_\infty\|\nabla d\|_\infty)ds\\
\leq&C[(t-t_0)^{-1/2}+(t-t_0)^{1/2}]+C\int_{t_0}^t(t-s)^{-1/2}(\|\nabla^{l_0+2}d\|_\infty+\|\nabla^{l_0+1}u\|_\infty)ds
\end{align*}
for any $t_0<t\leq T$. Hence, summing the above two inequalities up, we obtain
\begin{align*}
&(\|\nabla^{l_0+1}u\|_\infty+\|\nabla^{l_0+2}d\|_\infty)(t)\\
\leq& C[(t-t_0)^{-1/2}+(t-t_0)^{1/2}]+C\int_{t_0}^t(t-s)^{-1/2}(\|\nabla^{l_0+1}u\|_\infty+\|\nabla^{l_0+2}d\|_\infty)ds
\end{align*}
for any $t_0<t\leq T$, which gives
\begin{align*}
&(t-t_0)^{1/2}(\|\nabla^{l_0+1}u\|_\infty+\|\nabla^{l_0+2}d\|_\infty)(t)\\
\leq&C(1+t-t_0)+C(t-t_0)^{1/2}\int_{t_0}^t(t-s)^{-1/2}(\|\nabla^{l_0+1}u\|_\infty+\|\nabla^{l_0+1}d\|_\infty)ds,\quad t>t_0.
\end{align*}
Setting $\phi(t)=\sup_{t_0<s\leq t}(s-t_0)^{1/2}(\|\nabla^{l_0+1}u\|_\infty+\|\nabla^{l_0+2}d\|_\infty)(s)$, then
\begin{align*}
\phi(t)\leq&C(1+t-t_0)+C(t-t_0)^{1/2}\int_{t_0}^t(t-s)^{-1/2}(s-t_0)^{-1/2}ds\phi(t)\\
=&C(1+t-t_0)+C(t-t_0)^{1/2}\phi(t)\leq C(1+t-t_0)+\frac{1}{2}\phi(t),\qquad t_0<t\leq t_0+\sigma
\end{align*}
for some positive constant $\sigma<\delta$ depending only on $\delta, n, k$ and $M$, and thus
$$
\phi(t)\leq C,\qquad t_0<t\leq t_0+\sigma.
$$
This inequality implies
$$
(\|\nabla^{l_0+1}u\|_\infty+\|\nabla^{l_0+2}d\|_\infty)(t)\leq C(t-t_0)^{-1/2},\qquad t_0<t\leq t_0+\sigma,
$$
and thus
$$
(\|\nabla^{l_0+1}u\|_\infty+\|\nabla^{l_0+2}d\|_\infty)\leq C,\qquad t_0+\frac{\sigma}{2}<t\leq t_0+\sigma.
$$
Let $t_0$ take values over interval $[\frac{\delta}{2}, T-\sigma]$, recalling that $0<\sigma\leq\delta$, we obtain
$$
\sup_{\delta\leq t\leq T}(\|\nabla^{l_0+1}u\|_\infty+\|\nabla^{l_0+2}d\|_\infty)(t)\leq C,
$$
thus the estimates hold true for $l_0+1$, and finally we obtain
$$
\|(u, d)\|_{L^{\infty}(\delta, T; W^{k,\infty}\times W^{k+1,\infty})}\leq C.
$$

We now prove the time continuity. It follows
\begin{align*}
\|u(t)-u(t_0)\|_\infty=&\left\|(e^{(t-t_0)\Delta}-1)u(t_0)-\int_{t_0}^te^{(t-s)\Delta}\mathbb P\textmd{div}(u\otimes u+\nabla d\odot\nabla d)ds\right\|_\infty\\
\leq&\|(e^{(t-t_0)\Delta}-1)u(t_0)\|_\infty+\int_{t_0}^t\|e^{(t-s)\Delta}\mathbb P\textmd{div}(u\otimes u+\nabla d\odot\nabla d)\|_\infty ds\\
=&\left\|\int_0^{t-t_0}e^{s\Delta}\Delta u(t_0)ds\right\|_\infty+\int_{t_0}^t\|e^{(t-s)\Delta}\mathbb P\textmd{div}(u\otimes u+\nabla d\odot\nabla d)\|_\infty ds\\
\leq&\int_0^{t-t_0}\|\Delta u(t_0)\|_\infty ds+C\int_{t_0}^t(\|u\|_\infty^2+\|\nabla d\|_\infty^2)ds\\
\leq&C(|t-t_0|+|t-t_0|^{1/2}),\qquad \delta\leq t_0\leq t\leq T,
\end{align*}
and
\begin{align*}
\|d(t)-d(t_0)\|_\infty=&\left\|(e^{(t-t_0)\Delta}-1)d_0+\int_{t_0}^te^{(t-s)\Delta}(|\nabla d|^2d-(u\cdot\nabla)d)ds\right\|_\infty\\
\leq&\left\|\int_0^{t-t_0}e^{s\Delta}\Delta d_0 ds\right\|_\infty+\int_{t_0}^t\|e^{(t-s)\Delta}(|\nabla d|^2d-(u\cdot\nabla)d)\|_\infty ds\\
\leq&C\|\Delta d(t_0)\|_\infty|t-t_0|+C\int_{t_0}^t(\|\nabla d\|_\infty^2+\|u\|_\infty\|\nabla d\|_\infty)ds\\
\leq&C|t-t_0|,\qquad \delta\leq t_0\leq t\leq T.
\end{align*}
By the aid of the above two inequalities, it follows from the interpolation inequality that
\begin{align*}
\|\nabla^lu(t)-\nabla^lu(t_0)\|_\infty\leq& C\|u(t)-u(t_0)\|_\infty^{1-\frac{l}{k}}\|u(t)-u(t_0)\|_{W^{k,\infty}}^{\frac{l}{k}}\\
\leq& C(|t-t_0|^{1-\frac{l}{k}}+|t-t_0|^{\frac{1}{2}-\frac{l}{2k}}),\qquad 0\leq l\leq k-1,
\end{align*}
and
\begin{align*}
\|\nabla^ld(t)-\nabla^ld(t_0)\|_\infty\leq&C\|d(t)-d(t_0)\|_\infty^{1-\frac{l}{k+1}}\|d(t)-d(t_0)\|_{W^{k+1,\infty}}^{\frac{l}{k+1}}\\
\leq&C|t-t_0|^{1-\frac{l}{k+1}},\qquad 0\leq l\leq k,
\end{align*}
and consequently, we have
\begin{align*}
\|u(t)-u(t_0)\|_{W^{k-1,\infty}}\leq&C\sum_{l=0}^{k-1}\left(|t-t_0|^{1-\frac{l}{k}}+|t-t_0|^{\frac{1}{2}-\frac{l}{2k}}\right)\\
\leq&C(|t-t_0|+|t-t_0|^{\frac{1}{k}}+|t-t_0|^{\frac{1}{2}}+|t-t_0|^{\frac{1}{2k}})\\
\leq&C(|t-t_0|+|t-t_0|^{\frac{1}{2k}}),
\end{align*}
and
\begin{align*}
\|d(t)-d(t_0)\|_{W^{k,\infty}}\leq C\sum_{l=0}^k|t-t_0|^{1-\frac{l}{k+1}}\leq C(|t-t_0|+|t-t_0|^{\frac{1}{k+1}}).
\end{align*}
 The proof is complete.
\end{proof}

\section{Estimates on the existence time in terms of $L^\infty$ norm of initial data}\label{sec4}

In this section, we estimates the lower bound of the existence time of the mild solutions with smooth initial data in terms of the $L^\infty$ norm of the initial data. That's the following Proposition.

\begin{proposition}\label{prop4.1}
Let $k\geq2$, and$(u, d)\in X_{T_0}^k$ be the solution obtained in Proposition \ref{prop2.1}, then the existence time $T_0$ can be chosen such that
$$
T_0\geq C(\|u_0\|_\infty+\|\nabla d_0\|_\infty)^{-2}
$$
for some positive constant $C$ depending only on $n$, and
$$
\|(u,\nabla d)\|_{L^\infty(0, T_0; L^\infty)}\leq 2(\|u_0\|_\infty+\|\nabla d_0\|_\infty).
$$
\end{proposition}

\begin{proof}
We first estimate the lower bound of the existence time. By Proposition \ref{prop2.1}, system (\ref{1.1})--(\ref{1.4}) has a local mild solution $(u, d)$. Let's extend such solution to the maximal existence time $T^*$. If $T^*=\infty$, then we are down. Hence, we suppose that $T^*<\infty$, and then by Proposition \ref{prop2.1}, the maximal existence time $T^*$ can be characterized as
$$
\lim_{T\rightarrow T^*}\|(u,d)\|_{X_T^k}=\infty\qquad \mbox{ and }\qquad\|(u, d)\|_{X_T^k}<\infty\quad\mbox{ for any }T<T^*.
$$
It follows from Lemma \ref{lem2.1} that
\begin{align*}
\|u(t)\|_\infty=&\left\|e^{t\Delta}u_0-\int_0^te^{(t-s)\Delta}\mathbb P\textmd{div}(u\otimes u+\nabla d\odot\nabla d)ds\right\|\\
\leq&C\|u_0\|_\infty+C\int_0^t(t-s)^{-1/2}(\|u\|_\infty^2+\|\nabla d\|_\infty^2)ds,
\end{align*}
and
\begin{align*}
\|\nabla d(t)\|_\infty=&\left\|e^{t\Delta}\nabla d_0+\int_0^t\nabla e^{(t-s)\Delta}(|\nabla d|^2d-(u\cdot\nabla)d)ds\right\|_\infty\\
\leq&C\|\nabla d_0\|_\infty+C\int_0^t(t-s)^{-1/2}(\|\nabla d\|_\infty^2+\|u\|_\infty^2)ds,
\end{align*}
and thus
$$
(\|u\|_\infty+\|\nabla d\|_\infty)(t)\leq C(\|u_0\|_\infty+\|\nabla d_0\|_\infty)+C\int_0^t(t-s)^{-1/2}(\|u\|_\infty^2+\|\nabla d\|_\infty^2)ds
$$
for any $0\leq t<T^*$. Setting $f(t)=\sup_{0\leq s\leq t}(\|u\|_\infty+\|\nabla d\|_\infty)(s)$, then there holds
\begin{equation}\label{4.0}
f(t)\leq C_*(\|u_0\|_\infty+\|\nabla d_0\|_\infty)+C_*t^{1/2}f(t)^2
\end{equation}
for any $0<t<T^*$, where $C_*$ is a positive constant depending only on $n$.

We claim that $T^*>T_0$, where $T_0$ is given by
\begin{equation*}
T_0=\left(\frac{1}{4C_*(\|u_0\|_\infty+\|\nabla d_0\|_\infty)}\right)^2.
\end{equation*}
Suppose that $T^*\leq T_0$, then the inequality (\ref{4.0}) is equivalent to
$$
f(t)\geq \bar f(t)\qquad \mbox{or}\qquad f(t)\leq \underline f(t)
$$
where
$$
\bar f(t)=\frac{1+\sqrt{1-4C_*t^{1/2}(\|u_0\|_\infty+\|\nabla d_0\|_\infty)}}{2C_*t^{1/2}}
$$
and
$$
\underline f(t)=\frac{1-\sqrt{1-4C_*t^{1/2}(\|u_0\|_\infty+\|\nabla d_0\|_\infty)}}{2C_*t^{1/2}}.
$$
By Proposition \ref{prop2.1} and Proposition \ref{prop3.1}, it follows that $f(t)$ is bounded and continuous on $(0, T]$ for any $0<T<T^*$, and
$$
\lim_{t\rightarrow0^+}(\|u\|_\infty+\|\nabla d\|_\infty)(t)=\|u_0\|_\infty+\|\nabla d_0\|_\infty.
$$
One can easily check that
$$
\bar f(t)>\underline f(t), \quad \lim_{t\rightarrow 0^+}\bar f(t)=\infty\quad\mbox{and}\quad\lim_{t\rightarrow 0^+}\underline f(t)=\|u_0\|_\infty+\|\nabla d_0\|_\infty.
$$
These facts force $f(t)$ to satisfies $f(t)\leq\underline f(t)$, and thus, noticing that $\underline f(t)$ is increasing in $(0, T_0]$, we have
$$
f(T)\leq\underline f(T)\leq \underline f(T_0)=2(\|u_0\|_\infty+\|\nabla d_0\|_\infty),
$$
which gives
\begin{equation*}
\|(u,\nabla d)\|_{L^\infty(0, T; L^\infty)}\leq 2(\|u_0\|_\infty+\|\nabla d_0\|_\infty),\quad\forall 0\leq T<T^*.
\end{equation*}
Now we can apply Proposition \ref{prop3.1} to deduce
$$
\|(u, \nabla d)\|_{L^\infty(\delta, T; W^{k,\infty})}\leq C\quad\mbox{and}\quad\|(u,\nabla d)\|_{C^\alpha([\delta, T]; W^{k-1,\infty})}\leq C
$$
with $C$ independent of $T$. On account of this inequality, by taking $T\rightarrow T^*$, one can extend $(u, d)$ continuously to be defined on $[0,T^*]$, such that $$
\|(u,\nabla d)\|_{L^\infty(\delta, T^*; W^{k,\infty})}\leq C\quad\mbox{and}\quad\|(u,\nabla d)\|_{C^\alpha([\delta, T^*]; W^{k-1,\infty})}\leq C.
$$
By Proposition \ref{prop2.1}, we can extend $(u, d)$ to be a mild solution on $[0, T^{**}]$ for some $T^{**}>T^*$, which contradicts to the definition of $T^*$. This contradiction provides us that $T^*>T_0$, and $(u, d)$ is a mild solution to system (\ref{1.1})--(\ref{1.4}) on $[0, T_0]$.

Now we prove that the mild solution $(u, d)$ defined on $[0, T_0]$ satisfies
$$
\|(u, \nabla d)\|_{L^\infty(0, T_0; L^\infty)}\leq 2(\|u_0\|_\infty+\|\nabla d_0\|_\infty).
$$
In fact, starting from (\ref{4.0}), we can use the same procedure as in the above paragraph to obtain such estimate. Thus we omit its proof here.
\end{proof}

\section{Well-posedness and regularity with $L^\infty$ initial data}\label{sec5}

In this section, we prove the local existence and uniqueness of mild solutions to the system (\ref{1.1})--(\ref{1.4}) with $L^\infty$ initial data, and we also prove the regularity of such mild solutions, in other words, we prove Theorem \ref{thm1.1}.

\textbf{Proof of Theorem \ref{thm1.1}.} We first prove the local existence. Take $(u_0^\varepsilon, d_0^\varepsilon)$ such that
\begin{eqnarray*}
&u_0^\varepsilon\in W^{k,\infty},\quad d_0^\varepsilon\in W^{k+1,\infty},\quad\|u_0^\varepsilon\|_\infty\leq\|u_0\|_\infty,\quad\|d_0^\varepsilon\|_\infty\leq\|d_0\|_\infty,\quad\|\nabla d_0^\varepsilon\|_\infty\leq\|\nabla d_0\|_\infty,\\
&(u_0^\varepsilon(x), d_0^\varepsilon(x), \nabla d_0^\varepsilon(x))\rightarrow(u_0(x), d_0(x), \nabla d_0(x))\quad\mbox{as }\varepsilon\rightarrow0, \mbox{a.e.}.
\end{eqnarray*}

By Proposition \ref{prop4.1}, for any $\varepsilon$, there is a mild solution $(u^\varepsilon, d^\varepsilon)$ to system (\ref{1.1})--(\ref{1.4}) with initial data $(u_0^\varepsilon, d_0^\varepsilon)$ on $[0, T_\varepsilon]$ with $T_\varepsilon\geq C_*(\|u_0^\varepsilon\|_\infty+\|\nabla d_0^\varepsilon\|_\infty)^{-2}\geq C_*(\|u_0\|_\infty+\|\nabla d_0\|_\infty)^{-2}$, such that
\begin{equation}\label{5.0-1}
\|(u^\varepsilon, \nabla d^\varepsilon)\|_{L^\infty(0, T_\varepsilon; L^\infty)}\leq 2(\|\nabla d_0\|_\infty+\|u_0\|_\infty).
\end{equation}
Without loss of generality, we can suppose that all these $(u^\varepsilon, d^\varepsilon)$ are defined on a common time interval $[0,T]$ with a certain
\begin{equation}\label{5.0}
T\geq C_*(\|\nabla d_0\|_\infty+\|u_0\|_\infty)^{-2}.
\end{equation}

By Proposition \ref{prop3.1}, it follows from (\ref{5.0-1}) that
\begin{eqnarray}
& \|(u^\varepsilon, d^\varepsilon)\|_{L^\infty(\delta, T; W^{k,\infty}\times W^{k+1,\infty})}\leq C,\qquad \|(u^\varepsilon, d^\varepsilon)\|_{C^\alpha([\delta, T]; W^{k-1,\infty}\times W^{k,\infty})}\leq C\label{5.0-2}
\end{eqnarray}
for a certain $\alpha\in(0, 1)$ and for any $0<\delta<T$, where $C$ is a positive constant depending only on $\delta, k, n$ and $\|u_0\|_\infty+\|\nabla d_0\|_\infty$. By the aid of (\ref{5.0-1}) and (\ref{5.0-2}), using diagonal argument and applying Arzela-Ascoli theorem, there is a subsequence of $(u^\varepsilon, d^\varepsilon)$, still denoted by $(u^\varepsilon, d^\varepsilon)$, and $(u, d)$, such that
$$
(u^\varepsilon, d^\varepsilon, \nabla d^\varepsilon)\rightarrow(u, d, \nabla d)
$$
pointwisely, and
\begin{eqnarray}
& \|(u, \nabla d)\|_{L^\infty(0, T; L^\infty)}\leq C,\quad \|(u, d)\|_{L^\infty(\delta, T; W^{k,\infty}\times W^{k+1,\infty})}\leq C,\label{5.4}
\end{eqnarray}
for any $0<\delta<T$.

We claim that all the following weak star limits in $L^\infty$ hold true
\begin{eqnarray}
&(u^\varepsilon(t), d^\varepsilon(t))\rightarrow(u(t), d(t)), \qquad (e^{t\Delta}u_0^\varepsilon, e^{t\Delta}d_0^\varepsilon)\rightarrow(e^{t\Delta}u_0, e^{t\Delta}d_0),\label{7.1}\\
&\int_0^te^{(t-s)\Delta}(|\nabla d^\varepsilon|^2d^\varepsilon-(u^\varepsilon\cdot\nabla)d^\varepsilon)ds\rightarrow\int_0^te^{(t-s)\Delta}(|\nabla d|^2d-(u\cdot\nabla)d)ds,\label{7.2}\\
&\int_0^te^{(t-s)\Delta}\mathbb P\textmd{div}(u^\varepsilon\otimes u^\varepsilon+\nabla d^\varepsilon\odot\nabla d^\varepsilon)ds\rightarrow\int_0^te^{(t-s)\Delta}\mathbb P\textmd{div}(u\otimes u+\nabla d\odot\nabla d)ds.\label{7.3}
\end{eqnarray}
Take arbitrary $\varphi\in C_0^\infty(\mathbb R^n)$, by Lemma \ref{lem2.1}, there holds $\|e^{t\Delta}\varphi\|_1\leq C\|\varphi\|_1$ for any $t>0$. Recalling that
$$
(u^\varepsilon, d^\varepsilon, \nabla d^\varepsilon)\rightarrow(u, d, \nabla d)
$$
pointwisely, $(u_0^\varepsilon, d_0^\varepsilon)(x)\rightarrow(u_0, d_0)(x)$, a.e. $x\in\mathbb R^n$, and $\|(u^\varepsilon, d^\varepsilon)\|_{L^\infty(0, T; L^\infty\times W^{1,\infty})}\leq C$, it follows from Lebesgue's dominate convergence theorem that
\begin{eqnarray*}
&\int_{\mathbb R^n}(u^\varepsilon(x, t), d^\varepsilon(x, t))\varphi(x)dx\rightarrow\int_{\mathbb R^n}(u(x,t), d(x,t))\varphi(x)dx,\\
&\int_{\mathbb R^n}e^{t\Delta}u_0^\varepsilon\varphi dx=\int_{\mathbb R^n}u_0^\varepsilon e^{t\Delta}\varphi dx\rightarrow\int_{\mathbb R^n}u_0e^{t\Delta}\varphi dx=\int_{\mathbb R^n}e^{t\Delta}u_0\varphi dx,\\
&\int_{\mathbb R^n}e^{t\Delta}d_0^\varepsilon\varphi dx=\int_{\mathbb R^n}d_0^\varepsilon e^{t\Delta}\varphi dx\rightarrow\int_{\mathbb R^n}d_0e^{t\Delta}\varphi dx=\int_{\mathbb R^n}e^{t\Delta}d_0\varphi dx,
\end{eqnarray*}
and
\begin{align*}
&\int_{\mathbb R^n}\left(\int_0^t e^{(t-s)\Delta}(|\nabla d^\varepsilon|^2d^\varepsilon-(u^\varepsilon\cdot\nabla)d^\varepsilon)ds\right)\varphi(x)dx\\
=&\int_0^t\int_{\mathbb R^n}e^{(t-s)\Delta}\varphi(|\nabla d^\varepsilon|^2d^\varepsilon-(u^\varepsilon\cdot\nabla)d^\varepsilon)dxds\\
\rightarrow&\int_0^t\int_{\mathbb R^n}e^{(t-s)\Delta}\varphi(|\nabla d|^2d-(u\cdot\nabla)d)dxds\\
=&\int_0^t\int_{\mathbb R^n}e^{(t-s)\Delta}(|\nabla d|^2d-(u\cdot\nabla)d)\varphi dxds\\
=&\int_{\mathbb R^n}\left(\int_0^t e^{(t-s)\Delta}(|\nabla d|^2d-(u\cdot\nabla)d)ds\right)\varphi(x)dx,
\end{align*}
thus (\ref{7.1}) and (\ref{7.2}) hold true.

Now, we turn to the proof of (\ref{7.3}). Given $t>0$, then for any $0<t_0<t$, it follows
\begin{align*}
I^\varepsilon=&\left|\int_{\mathbb R^n}\left(\int_0^t e^{(t-s)\Delta}\mathbb P\textmd{div}(u^\varepsilon\otimes u^\varepsilon+\nabla d^\varepsilon\odot\nabla d^\varepsilon-u\otimes u-\nabla d\odot\nabla d)ds\right)\varphi(x)dx\right|\\
=&\left|\int_0^t\int_{\mathbb R^n}e^{(t-s)\Delta}\mathbb P\textmd{div}(u^\varepsilon\otimes u^\varepsilon+\nabla d^\varepsilon\odot\nabla d^\varepsilon-u\otimes u-\nabla d\odot\nabla d)\varphi(x)dxds\right|\\
\leq&\int_0^{t-t_0}\int_{\mathbb R^n}|e^{(t-s)\Delta}\mathbb P\nabla \varphi||u^\varepsilon\otimes u^\varepsilon+\nabla d^\varepsilon\odot\nabla d^\varepsilon-u\otimes u-\nabla d\odot\nabla d|dxds\\
&+\int_{t-t_0}^t\int_{\mathbb R^n}|e^{(t-s)\Delta}\mathbb P\nabla \varphi||u^\varepsilon\otimes u^\varepsilon+\nabla d^\varepsilon\odot\nabla d^\varepsilon-u\otimes u-\nabla d\odot\nabla d|dxds\\
=&I_1^\varepsilon(t_0)+I_2^\varepsilon(t_0).
\end{align*}
We will show that $I_2^\varepsilon(t_0)\rightarrow 0$ as $t_0\rightarrow0$, uniformly with respective to $\varepsilon$, and $I_1^\varepsilon(t_0)\rightarrow0$ as $\varepsilon\rightarrow0$ for each fixed $t_0\in(0, t)$. We first consider $I_2^\varepsilon(t_0)$. Recalling that $\|(u^\varepsilon, d^\varepsilon)\|_{L^\infty(0, T; L^\infty\times W^{1,\infty})}\leq C$ and $\|(u, d)\|_{L^\infty(0, T; L^\infty\times W^{1,\infty})}\leq C$, it follows from Lemma \ref{lem2.1} that
\begin{align}
I_2^\varepsilon(t_0)\leq&\int_{t-t_0}^t\int_{\mathbb R^n}|e^{(t-s)\Delta}\mathbb P\nabla\varphi|dxds\leq C\int_{t-t_0}^t(t-s)^{-1/2}\|\varphi\|_1ds\nonumber\\
\leq&Ct_0^{1/2}\|\varphi\|_1\rightarrow0,\qquad\mbox{as }t_0\rightarrow0,\mbox{ uniformly w.r.t. }\varepsilon.\label{7.5}
\end{align}
Next, we consider $I_1^\varepsilon(t_0)$ for fixed $t_0$. Define
$$
f_\varepsilon(s)=\int_{\mathbb R^n}|e^{(t-s)\Delta}\mathbb P\nabla \varphi||u^\varepsilon\otimes u^\varepsilon+\nabla d^\varepsilon\odot\nabla d^\varepsilon-u\otimes u-\nabla d\odot\nabla d|dx,\quad s\in[0, t-t_0].
$$
By Lemma \ref{lem2.1}, recalling that $u^\varepsilon, \nabla d^\varepsilon, u$ and $\nabla d$ are all bounded, we have
\begin{equation}\label{7.6}
|f_\varepsilon(s)|\leq C\|e^{(t-s)\Delta}\mathbb P\nabla \varphi\|_1\leq C(t-s)^{-1/2}\|\varphi\|_1\leq Ct_0^{-1/2}\|\varphi\|_1,\quad\forall s\in[0, t-t_0].
\end{equation}
Moreover, for each $s\in[0, t-t_0]$, since $e^{(t-s)\Delta}\mathbb P\nabla\varphi\in L^1(\mathbb R^n)$ and $(u^\varepsilon, \nabla d^\varepsilon)(x, t)\rightarrow(u, \nabla d)(x, t)$ a.e. $x\in\mathbb R^3$, it follows from Lebesgue's dominate convergence theorem that $$
f_\varepsilon(s)\rightarrow0,\qquad\forall s\in[0, t-t_0].
$$
This combined with (\ref{7.6}), it follows from Lebesgue's dominate convergence theorem again that
\begin{equation}\label{7.7}
I_1^\varepsilon(t_0)=\int_0^{t-t_0}f_\varepsilon(s)ds\rightarrow0,\qquad\mbox{ as }\varepsilon\rightarrow0^+.
\end{equation}
By the aid of (\ref{7.5}) and (\ref{7.7}), we can deduce $I^\varepsilon\rightarrow0$ as $\varepsilon\rightarrow0^+$ by first taking $t_0$ small and then letting $\varepsilon$ small to show that the quantity of $I^\varepsilon$ is arbitrary small as $\varepsilon\rightarrow0$, and thus (\ref{7.3}) holds true.

Since $(u^\varepsilon, d^\varepsilon)$ is a mild solution, it follows
\begin{eqnarray}
& u^\varepsilon(t)=e^{t\Delta}u_0^\varepsilon-\int_0^te^{(t-s)\Delta}\mathbb P\textmd{div}(u^\varepsilon\otimes u^\varepsilon+\nabla d^\varepsilon\odot\nabla d^\varepsilon)ds,\qquad 0\leq t\leq T,\label{5.1}\\
& d^\varepsilon(t)=e^{t\Delta}d_0^\varepsilon+\int_0^te^{(t-s)\Delta}(|\nabla d^\varepsilon|^2d^\varepsilon-(u^\varepsilon\cdot\nabla)d^\varepsilon)ds,\qquad0\leq t\leq T.\label{5.2}
\end{eqnarray}
By the aid of (\ref{7.1})--(\ref{7.3}), we can take the weak star limit in (\ref{5.1}) and (\ref{5.2}) to conclude that
\begin{eqnarray*}
  u(t) &=& e^{t\Delta}u_0-\int_0^te^{(t-s)\Delta}\mathbb P\textmd{div}(u\otimes u+\nabla d\odot\nabla d)ds,\qquad0<t\leq T, \\
  d(t) &=& e^{t\Delta}d_0+\int_0^t e^{(t-s)\Delta}(|\nabla d|^2d-(u\cdot\nabla)d)ds,\qquad 0<t\leq T.
\end{eqnarray*}
These two identity automatically hold true at $t=0$. Thus $(u, d)$ is a mild solution to system (\ref{1.1})--(\ref{1.4}) on $[0, T]$.

Next, we prove the uniqueness of bounded local mild solutions. Let $(u, d)$ and $(v, n)$ be mild solutions to system (\ref{1.1})--(\ref{1.4}) satisfying
$$
\|(u, d)\|_{L^\infty(0, T; L^\infty\times W^{1,\infty})}\leq M \qquad \|(v, n)\|_{L^\infty(0, T; L^\infty\times W^{1,\infty})}\leq M
$$
for some positive constant $M$.
Then we have
\begin{align*}
u(t)-v(t)=&\int_0^t e^{(t-s)\Delta}\mathbb P\textmd{div}(v\otimes v-u\otimes u+\nabla n\odot\nabla n-\nabla d\odot\nabla d)ds\\
=&\int_0^t e^{(t-s)\Delta}\mathbb P\textmd{div}((v-u)\otimes v+u\otimes(v-u)\\
&+\nabla(n-d)\odot\nabla n+\nabla d\odot\nabla(n-d))ds,
\end{align*}
and
\begin{align*}
d(t)-n(t)=&\int_0^t e^{(t-s)\Delta}(|\nabla d|^2d-|\nabla n|^2n+(v\cdot\nabla)n-(u\cdot\nabla)d)ds\\
=&\int_0^te^{(t-s)\Delta}(\nabla(d-n)\nabla(d+n)d+|\nabla n|^2(d-n)\\
&+(v-u)\nabla n+u\nabla(n-d))ds,
\end{align*}
and thus it follows from Lemma \ref{lem2.1} that
\begin{eqnarray*}
&\|u-v\|_\infty(t)\leq C\int_0^t(t-s)^{-1/2}(\|u-v\|_\infty+\|\nabla n-\nabla d\|_\infty)ds,\\
&\|d-n\|_\infty(t)\leq C\int_0^t(\|\nabla d-\nabla n\|_\infty+\|d-n\|_\infty+\|u-v\|_\infty)ds,\\
&\|\nabla d-\nabla n\|_\infty(t)\leq C\int_0^t(t-s)^{-1/2}(\|\nabla d-\nabla n\|_\infty+\|d-n\|_\infty+\|u-v\|_\infty)ds.
\end{eqnarray*}
Combining these inequalities, we obtain
$$
(\|u-v\|_\infty+\|d-n\|_{W^{1,\infty}})(t)\leq C\int_0^t(t-s)^{-1/2}(\|u-v\|_\infty+\|d-n\|_{W^{1,\infty}})ds.
$$
Setting $\phi(t)=\sup_{0\leq s\leq t}(\|u-v\|_\infty+\|d-n\|_{W^{1,\infty}})(s)$, it follows
$$
\phi(t)\leq C\int_0^t(1+(t-s)^{-1/2})ds\phi(t)=C(t^{1/2}+t)\phi(t),
$$
from which we obtain
$$
\phi(t)=0,\qquad 0\leq t\leq\sigma,
$$
where $\sigma$ is a positive constant depending only on $n$ and $M$, and consequently
$$
(u, d)=(v, n),\qquad t\in [0, \sigma].
$$
Similarly, we can prove that $(u, d)=(v, n)$ on $[\sigma, 2\sigma]$, and finally, we obtain $(u, d)=(v, n)$ on $[0, T]$. This completes the proof of the uniqueness.

Then, we prove the regularity (i). Since $(u, d)$ is a mild solution, there hold
\begin{eqnarray*}
  u(t) &=& e^{(t-t_0)\Delta}u(t_0)-\int_{t_0}^te^{(t-s)\Delta}\mathbb P\textmd{div}(u\otimes u+\nabla d\odot\nabla d)ds,\qquad0<t\leq T, \\
  d(t) &=& e^{(t-t_0)\Delta}d(t_0)+\int_{t_0}^t e^{(t-s)\Delta}(|\nabla d|^2d-(u\cdot\nabla)d)ds,\qquad 0<t\leq T,
\end{eqnarray*}
and thus, for any $\delta\leq t_0\leq t\leq T$, recalling (\ref{5.4}), we have
\begin{align*}
&\|\nabla^k(u(t)-u(t_0))\|_\infty\\
\leq&\|(e^{(t-t_0)\Delta}-1)\nabla^ku(t_0)\|_\infty+\int_0^t\|e^{(t-s)\Delta}\mathbb P\textmd{div}\nabla^k(u\otimes u+\nabla d\odot\nabla d)\|_\infty ds\\
\leq&C\int_0^{t-t_0}\|e^{s\Delta}\Delta\nabla^ku(t_0)\|_\infty ds+C\int_0^t(t-s)^{-1/2}\|\nabla^k(u\otimes u+\nabla d\odot\nabla d)\|_\infty ds\\
\leq&C\|\Delta\nabla^k u(t_0)\|_\infty|t-t_0|+C\int_{t_0}^t(t-s)^{-1/2}\|(u,\nabla d)\|_{W^{k,\infty}}ds\\
\leq&C(|t-t_0|+|t-t_0|^{1/2})\leq C|t-t_0|^{1/2},
\end{align*}
and
\begin{align*}
&\|\nabla^{k+1}(d(t)-d(t_0))\|_\infty\\
\leq&\left\|(e^{(t-t_0)\Delta}-1)\nabla^{k+1}d(t_0)\right\|_\infty+\int_{t_0}^t\|e^{(t-s)\Delta}\nabla^{k+1}(|\nabla d|^2d-(u\cdot\nabla)d)\|_\infty ds\\
\leq&C\int_0^{t-t_0}\|e^{s\Delta}\Delta\nabla^{k+1}d(t_0)\|_\infty ds+C\int_{t_0}^t\|(u, d)\|_{W^{k+1,\infty}\times W^{k+2,\infty}}ds\\
\leq&C|t-t_0|.
\end{align*}
In the above, we have used the fact that
$$
\|(u, d)\|_{L^\infty(\delta, T; W^{k+2,\infty}\times W^{k+3,\infty})}\leq C,
$$
which is guaranteed by (\ref{5.4}), since it holds true for any $k$. Thus, we have proven that
$$
u\in C^{1/2}([\delta, T]; W^{k,\infty})\quad\mbox{ and }\quad d\in Lip([\delta, T]; W^{k+1,\infty}).
$$

Now, we prove (ii), the continuity of the functions $\|u(t)\|_\infty$ and $\|\nabla d(t)\|_\infty$. Thanks to (i), one can easily see that $\|u(t)\|_\infty$ and $\|\nabla d(t)\|_\infty$ are continuous in $(0, T]$. While the continuity at $t=0$, i.e.
$$
\lim_{t\rightarrow0^+}\|u(t)\|_\infty=\|u_0\|_\infty\qquad\mbox{and}\qquad\lim_{t\rightarrow0^+}\|\nabla d(t)\|_\infty=\|\nabla d_0\|_\infty
$$
can be proven in the same way as in Proposition \ref{prop2.1}, thus we omit it here.

For (iii), the lower bound of the existence time $T$, follows from (\ref{5.0}). Thus we complete the proof of Theorem \ref{thm1.1}.

\section{Vorticity direction blow up criterion}\label{sec6}

In this section, we prove the vorticity direction blow up criterion for type I mild solution to system (\ref{1.1})--(\ref{1.3}), in other words, we prove Theorem \ref{thm1.2} and Theorem \ref{thm1.3}.

\textbf{Proof of Theorem \ref{thm1.2}.} By Theorem \ref{thm1.1}, it suffices to show that
\begin{equation}\label{6.1}
\limsup_{t\rightarrow 0^+}(\|u\|_\infty+\|\nabla d\|_\infty)(t)<\infty.
\end{equation}

We divide the proof of (\ref{6.1}) into four steps: in step 1, we use the blow up argument to derive a bounded backward mild solution $(\overline u, \overline d)$ to the system (\ref{1.1})--(\ref{1.3}); in step 2, using the continuity assumption on the direction filed, we prove that $\overline d$ is a constant vector field; in step 3, using the type I assumption, we prove that the vorticity field $\overline\omega\not\equiv0$; in the last step, step 4, we prove that $\overline\omega\equiv0$ by using the continuity assumption on the vorticity direction.

\textbf{Step 1. Blow up argument. }
Suppose that (\ref{6.1}) does not hold true, then
$$
\limsup_{t\rightarrow0^+}(\|u\|_\infty+\|\nabla d\|_\infty)(t)=\infty.
$$
Then, we can take $(x_k, t_k)\in\mathbb R^3\times(-1, 0)$ with $t_k\nearrow0$, such that
$$
M_k\overset{\textmd{def}}{=}\sup_{-1\leq t\leq t_k}(\|u\|_\infty+\|\nabla d\|_\infty)(t)\nearrow\infty,\qquad |u(x_k,t_k)|+|\nabla d(x_k, t_k)|\geq M_k-1.
$$
Define $(u_k, d_k)$ as follows
$$
u_k(x,t)=\frac{1}{M_k}u\left(x_k+\frac{x}{M_k}, t_k+\frac{t}{M_k^2}\right),\quad d_k=d\left(x_k+\frac{x}{M_k}, t_k+\frac{t}{M_k^2}\right)
$$
for any $(x, t)\in Q_k\overset{\textmd{def}}{=}\mathbb R^3\times(-(1+t_k)M_k^2, -t_kM_k^2)$. Since $(u, d)$ is a mild solution in $\mathbb R^3\times(-1, 0)$, it's easy to check that $(u_k, d_k)$ is a mild solution in $Q_k$, and
\begin{equation}\label{6.2}
1-\frac{1}{M_k}\leq|u_k(0, 0)|+|\nabla d_k(0, 0)|\leq 1.
\end{equation}
Noticing that $\|(u_k, \nabla d_k)\|_{L^\infty(Q_k)}\leq 1$, it follows from Proposition \ref{prop3.1} that for large $k$
$$
\|(u_k, d_k)\|_{L^\infty\left(\left[-M_k^2/2, 0\right]; W^{3,\infty}\times W^{4,\infty}\right)}\leq C,
$$
$$
\|u_k(t)-u_k(t_0)\|_{W^{2,\infty}}\leq C(|t-t_0|+|t-t_0|^{1/6}),
$$
and
$$
\|d_k(t)-d_k(t_0)\|_{W^{3,\infty}}\leq C(|t-t_0|+|t-t_0|^{1/4})
$$
for any $t_0, t\in[-\frac{M_2^2}{2}, 0]$,
where $C$ is a constant depending only on $n$. By Arzela-Ascoli theorem, there is a subsequence, still denoted by $(u_k, d_k)$, and $(\bar u, \bar d)$, such that
\begin{eqnarray}
&(\bar u, \bar d)\in C_{loc}^\alpha((-\infty, 0]; W^{2,\infty}\times W^{3,\infty}),\qquad \|(\bar u, \bar d)\|_{L^\infty(-\infty, 0; W^{2,\infty}\times W^{3,\infty})}\leq C,\nonumber\\
&|\bar u(0, 0)|+|\nabla \bar d(0, 0)|=1,\nonumber\\
&(u_k, d_k)\rightarrow(\bar u, \bar d)\quad\mbox{ in }C^{\alpha/2}_{loc}((-\infty, 0]; W^{2,\infty}(K)\times W^{3,\infty}(K))\label{6.3}
\end{eqnarray}
for a certain $\alpha\in(0, 1)$ and for any compact subset $K\subseteq\mathbb R^3$. Since $(u_k, d_k)$ is a mild solution, we have
\begin{eqnarray}
&u_k(t)=e^{(t-t_0)\Delta}u_k(t_0)-\int_{t_0}^te^{(t-s)\Delta}\mathbb P\textmd{div}(u_k\otimes u_k+\nabla d_k\odot\nabla d_k)ds,\label{6.4}\\
&d_k(t)=e^{(t-t_0)\Delta}d_k(t_0)+\int_{t_0}^te^{(t-s)\Delta}(|\nabla d_k|^2d_k-(u_k\cdot\nabla)d_k)ds.\label{6.5}
\end{eqnarray}
By the aid of (\ref{6.3}), one can prove (in the same way as what we done in the proof of (\ref{7.1})--(\ref{7.3})) that all the terms in (\ref{6.4}) and (\ref{6.5}) weak star converge to the corresponding terms of $(\bar u, \bar d)$, respectively, and thus $(\bar u, \bar d)$ is a mild solution to system (\ref{1.1})--(\ref{1.3}) in $\mathbb R^3\times(-\infty, 0]$.

\textbf{Step 2. $\bar d(\cdot, t)\equiv C$ in $\mathbb R^3\times(-\infty, 0]$.}
Take arbitrary $x, y\in \mathbb R^3$. It follows that
\begin{align*}
|d_k(x, t)-d_k(y, t)|=&\left|d\left(x_k+\frac{x}{M_k}, t_k+\frac{t}{M_k^2}\right)-d\left(x_k+\frac{y}{M_k}, t_k+\frac{t}{M_k^2}\right)\right|\\
\leq&\eta\left(\frac{|x-y|}{M_k}\right),
\end{align*}
which, by taking $k\rightarrow\infty$, gives $\bar d(x, t)-\bar d(y, t)=0$, and thus $\bar d(\cdot, t)\equiv C(t)$. Since $(\bar u, \bar d)$ is a mild solution, we have
$$
\bar d(t)=e^{(t-t_0)\Delta }\bar d(t_0)+\int_{t_0}^te^{(t-s)\Delta}(|\nabla\bar d|^2\bar d-(\bar u\cdot\nabla)\bar d)ds=e^{(t-t_0)\Delta} C(t_0)=C(t_0),
$$
and thus $\bar d\equiv C$ in $\mathbb R^3\times(-\infty, 0]$.

\textbf{Step 3. $\bar\omega\not\equiv0$ in $\mathbb R^3\times(-\infty, 0]$, where $\bar\omega=\textmd{curl} \bar u$.}
Suppose that $\bar\omega\equiv0$, then it follows
$$
\Delta\bar u=\nabla\textmd{div}\bar u-\textmd{curl curl}\bar u=0.
$$
The Liouville theorem for bounded harmonic functions in whole space yields $\bar u=C(t)$. Recalling that $\bar d\equiv C$ and $(\bar u, \bar d)$ is a mild solution to system (\ref{1.1})--(\ref{1.3}), we deduce
\begin{align*}
\bar u(t)=&e^{(t-s)\Delta}\bar u(t_0)-\int_{t_0}^te^{(t-s)\Delta}\mathbb P\textmd{div}(\bar u\otimes\bar u+\nabla\bar d\odot\nabla\bar d)ds\\
=&e^{(t-s)\Delta}\bar C(t_0)-\int_{t_0}^te^{(t-s)\Delta}\mathbb P(\bar u\cdot\nabla\bar u)ds\\
=&e^{(t-t_0)\Delta}C(t_0)=C(t_0),
\end{align*}
and thus $\bar u\equiv C$ in $\mathbb R^3\times(-\infty, 0]$. Recalling that $\bar d\equiv C$ in $\mathbb R^3\times(-\infty, 0]$ and $|\bar u(0, 0)|+|\nabla \bar d|(0, 0)=1$, we conclude that $|\bar u|\equiv1$ in $\mathbb R^3\times(-\infty, 0]$. Since $(u, d)$ is type I mild solution, there holds
\begin{align*}
\|u_k(t)\|_\infty=&\left\|\frac{1}{M_k}u\left(x_k+\frac{x}{M_k}, t_k+\frac{t}{M_k^2}\right)\right\|_\infty\leq CM_k^{-1}\left(-\left(t_k+\frac{t}{M_k^2}\right)\right)^{-1/2}\\
=&CM_k^{-1}\left(|t_k|+\frac{|t|}{M_k^2}\right)^{-1/2}\leq CM_k^{-1}\left(\frac{|t|}{M_k^2}\right)^{-1/2}=C|t|^{-1/2}.
\end{align*}
Taking $k\rightarrow\infty$ yields
$$
1=|\bar u|\leq C|t|^{-1/2},\qquad\forall t\in (-\infty, 0),
$$
which is a contradiction.

\textbf{Step 4. $\bar\omega\equiv0$ in $\mathbb R^3\times(-\infty, 0]$.} We divide Step 4 into three steps as follows.

\textbf{Step 4.1. $\overline\zeta\equiv\overline{\zeta_0}(t)$. }Set $\Omega(t)=\left\{x\in \mathbb R^3|\bar\omega(x, t)\not=0\right\}$. Take arbitrary compact subset $K$ of $\Omega(t)$, then there is $\delta>0$ such that $|\bar\omega(x,t)|\geq\delta$ for all $x\in K$. Since $\omega_k\rightarrow\bar\omega$ locally uniformly, for large $k$ we have $|\omega_k(x, t)|\geq\frac{\delta}{2}$ for all $x\in K$, and thus
$$
\frac{\delta}{2}\leq|\omega_k(x, t)|=\frac{1}{M_k^2}\left|\omega\left(x_k+\frac{x}{M_k}, t_k+\frac{t}{M_k^2}\right)\right|,
$$
which implies $x_k+\frac{x}{M_k}\in\Omega_\sigma\left(t_k+\frac{t}{M_k^2}\right)$ for all $x\in K$. Consequently, we have
\begin{align*}
|\zeta_k(x, t)-\zeta_k(y, t)|=&\left|\zeta\left(x_k+\frac{x}{M_k}, t_k+\frac{t}{M_k^2}\right)-\zeta\left(x_k+\frac{y}{M_k}, t_k+\frac{t}{M_k^2}\right)\right|\\
\leq&\eta\left(\frac{|x-y|}{M_k}\right),
\end{align*}
where $\zeta_k=\frac{\omega_k}{|\omega_k|}$ and $\omega_k=\textmd{curl}u_k$. Taking $k\rightarrow\infty$ in the above inequality gives
$$
|\bar\zeta(x, t)-\bar\zeta(y, t)|=0,\qquad\forall x,y\in K,
$$
and thus $\bar\zeta(\cdot, t)\equiv\bar\zeta_0(K,t)$ on $K$. Since $K$ is a arbitrary compact subset of $\Omega(t)$, we conclude that $\bar\zeta(\cdot, t)\equiv\bar\zeta_0(t)$ on $\Omega(t)$. Hence
$$
\bar\omega(x,t)=|\bar\omega(x, t)|\bar\zeta_0(t).
$$

\textbf{Step 4.2 $\zeta_0(t)\equiv C$.} For this aim, take arbitrary $t_0\in(-\infty, 0)$. Since system (\ref{1.1})--(\ref{1.3}) is rotational invariant, we may assume that $\bar\omega(x, t_0)=(0, 0, \bar\omega_3(x, t_0)$ and $\bar\zeta_0(x, t_0)=(0, 0, 1)$ by rotation. At time $t_0$, since $(\textmd{curl}\bar\omega)_3=0$, it follows that
$$
\Delta\bar u_3=\partial_3\textmd{div}\bar u-(\textmd{curl curl}\bar u)_3=-(\textmd{curl}\bar\omega)_3=0,
$$
where $(\textmd{curl}\bar\omega)_3$ denotes the third exponent of the vector $\textmd{curl}\bar\omega$. By Liouville theorem of bounded harmonic functions in the whole space, we conclude that $\bar u_3$ is spatially constant at time $t_0$, i.e., $\bar u_3(\cdot, t_0)\equiv C_0$. Recall that $\bar\omega(x, t_0)=\bar\omega_2(x, t_0)=0$, it has
$$
\frac{\partial\bar u_3}{\partial x_2}-\frac{\bar u_2}{\partial x_3}=\frac{\partial\bar u_1}{\partial x_3}-\frac{\bar u_3}{\partial x_1}=0
$$
at time $t_0$, and thus $\bar u_1$ and $\bar u_2$ are independent of $x_3$ at time $t_0$. Recalling that $\bar d\equiv C$ in $\mathbb R^3\times(-\infty, 0]$, we observe that $\bar u$ satisfies the Navier-Stokes equations
$$
\partial_t\bar u+(\bar u\cdot\nabla)\bar u-\Delta\bar u+\nabla\bar p=0,\quad\textmd{div}\bar u=0\qquad\mbox{in }\mathbb R^3\times(-\infty, 0].
$$
Since $\bar u$ is $x_3$ independent at $t_0$, the local existence and uniqueness theorem of bounded mild solutions for Navier-Stokes equations \cite{GIGA99} (or see Theorem \ref{thm1.1} of the present paper) implies that the solution stays two-dimensional (independent of $x_3$) for $t\in[t_0, t_0+\varepsilon]$ and thus $\bar\zeta_0(x, t)=(0, 0, 1)$ for $t\in[t_0, t_0+\varepsilon]$. By taking $t_0$ over all values in $(-\infty, 0)$, we conclude that $\bar u$ is independent of $x_3$ for all $t$, and thus $\bar\zeta_0$ is independent of $t$, i.e. $\bar\zeta_0(x, t)\equiv(0, 0, 1)$.

\textbf{Step 4.3. $\overline\omega\equiv0$. }Starting from the observation that $\bar\zeta_0(, t)=(0, 0, 1)$, one can prove that
$$
\bar u(x, t)=(\bar u_1(x_1, x_2, t), \bar u_2(x_1, x_2, t), C_0)
$$
for each $t\in(-\infty, 0]$ as what we do for $t=t_0$ in the previous paragraph. Thus $\bar\omega_3(x_1, x_2, t)$ solves the two dimensional vorticity equation
$$
\partial_t\bar\omega_3-\Delta\bar\omega_3+(\bar u\cdot\nabla)\bar\omega_3=0\qquad\mbox{ in }\mathbb R^2\times(-\infty, 0).
$$
Applying the Liouville type theorem (\cite{GIGA99} Lemma 2.3) to the above equation yields $\bar\omega_3\equiv0$, and thus $\bar\omega\equiv0$. This completes the proof of Step 4.

Obviously, the conclusion in Step 3 contradicts to that in Step 4. This contradiction provides us that (\ref{6.1}) holds true. This completes the proof of Theorem \ref{thm1.2}.~~~~~~~~~~~~~~~~~~~~~~~~~~~~~~~~~~$\square$

\textbf{Proof of Theorem \ref{thm1.3}.} Check the proof of Theorem \ref{thm1.2}, using the same notations as above, it suffices to show that $\overline\zeta(x, t)\equiv\overline{\zeta_0}(t)$, i.e. Step 4.1 in the above proof, because only in this step is the continuity assumption on the vorticity direction used.

Suppose that $\int_{-1}^0\|\nabla\zeta\|_{L^\beta(\Omega_\sigma(t))}^\alpha dt<\infty$ for some give $\sigma>0$ with $\frac{2}{\alpha}+\frac{3}{\beta}=1$ and $2\leq \alpha<\infty$. We calculate
\begin{align*}
\int_{-1}^0\|\nabla\zeta\|_{L^\beta(\Omega_\sigma(t))}^\alpha dt=&\int_{-1}^0\left(\int_{\Omega_\sigma(t)}|\nabla\zeta(x,t)|^\beta dx\right)^{\alpha/\beta}dt\\
=&\int_{-1}^0\left(\int_{\mathbb R^3}|\nabla\zeta(x,t)\chi_{\Omega_\sigma(t)}(x,t)|^\beta dx\right)^{\alpha/\beta}dt,
\end{align*}
where $\chi_E$ is the characteristic function of the set $E$.
For any $\tau>0$, we define
$$
f_\tau(x, t)=|\nabla\zeta(x, t)\chi_{\Omega_\tau(t)}(x,t)|^\beta,
$$
then for any fixed $t$, noticing that the set $\Omega_\tau(t)$ decreases to the empty set as $\tau$ goes to infinity, one can easily infer that $f_\tau(\cdot,t)$ decreases to $0$ as $\tau$ goes to infinity. Obviously there holds that
$$
0\leq f_{\tau}(x, t)\leq f_{\sigma}(x,t),\quad \forall\tau\geq\sigma.
$$
Dominate convergence theorem provides us that
\begin{align*}
\int_{-1}^0\|\nabla\zeta\|_{L^\beta(\Omega_\tau(t))}^\alpha dt=\int_{-1}^0\left(\int_{\mathbb R^3}|\nabla\zeta(x,t)\chi_{\Omega_\tau(t)}(x,t)|^\beta dx\right)^{\alpha/\beta}dt\rightarrow0,\mbox{ as }\tau\rightarrow\infty.
\end{align*}
On account of this fact, for any given $\varepsilon>0$, there is $\sigma_\varepsilon>0$, such that
$$
\int_{-1}^0\|\nabla\zeta\|_{L^\beta(\Omega_{\sigma_\varepsilon}(t))}^\alpha dt\leq\varepsilon.
$$
Set
$$
\Omega=\{(x, t)|\bar\omega(x, t)\not=0\},
$$
and take arbitrary compact set $K$ contained in $\Omega$. Recalling that $\bar\omega$ is continuous, there is a positive number $\delta$ such that
$$
|\bar\omega(x, t)|\geq\delta,\quad\forall(x, t)\in K.
$$
Recalling (\ref{6.3}), it's clear that $\omega_k\rightarrow\bar\omega$ uniformly on $K$, and thus
$$
|\omega_k(x,t)|\geq\frac{\delta}{2},\quad\forall(x,t)\in K,\mbox{ for large }k.
$$
Recalling that
$$
u_k(x,t)=\frac{1}{M_k}u(x_k+\frac{x}{M_k},t_k+\frac{t}{M_k^2}),
$$
it follows
$$
|\omega(x_k+\frac{x}{M_k},t_k+\frac{t}{M_k^2})|=|M_k^2\omega_k(x,t)|\geq \frac{\delta}{2}M_k^2\geq\sigma_\varepsilon,\quad\forall(x,t)\in K,\mbox{ for large }k,
$$
or equivalently
$$
x_k+\frac{K_t}{M_k}\subseteq\Omega_{\sigma_\varepsilon}(t_k+\frac{t}{M_k^2}),\quad\mbox{ for large }k,
$$
where $K_t=\{x\in\mathbb R^3|(x,t)\in K\}$. Note that
$$
\zeta_k(x,t)=\zeta(x_k+\frac{x}{M_k},t_k+\frac{t}{M_k^2}),\quad\nabla\zeta_k(x,t)=\nabla\zeta(x_k+\frac{x}{M_k},t_k+\frac{t}{M_k^2}).
$$
We calculate
\begin{align*}
&\int_{-(1+t_k)M_k^2}^{-t_kM_k^2}\|\nabla\zeta_k\|_{L^\beta(K_t)}^\alpha dt\\
=&\int_{-(1+t_k)M_k^2}^{-t_kM_k^2}\left(\int_{K_t}\left|\frac{1}{M_k}\nabla\zeta\left(x_k+\frac{x}{M_k},t_k+\frac{t}{M_k^2}\right)\right|^\beta dx\right)^{\alpha/\beta}dt\\
=&\int_{-(1+t_k)M_k^2}^{-t_kM_k^2}\left(\int_{x_k+\frac{K_t}{M_k}}\left|\frac{1}{M_k}\nabla\zeta\left(y, t_k+\frac{t}{M_k^2}\right)\right|^\beta M_k^3 dy\right)^{\alpha/\beta}dt\\
\leq&\int_{-(1+t_k)M_k^2}^{-t_kM_k^2}\left(\int_{\Omega_{\sigma_\varepsilon}\left(t_k+\frac{t}{M_k^2}\right)}\left|\frac{1}{M_k}\nabla\zeta\left(y, t_k+\frac{t}{M_k^2}\right)\right|^\beta M_k^3dy\right)^{\alpha/\beta}dt\\
=&\int_{-1}^{0}\left(\int_{\Omega_{\sigma_\varepsilon}\left(s\right)}M_k^{3-\beta}\left|\nabla\zeta\left(y, s\right)\right|^\beta dy\right)^{\alpha/\beta}M_k^2ds\\
=&\int_{-1}^{0}M_k^{2+\frac{\alpha}{\beta}(3-\beta)}\left(\int_{\Omega_{\sigma_\varepsilon}(s)}|\nabla\zeta(y, s)|^\beta dy\right)^{\alpha/\beta}M_k^2ds\\
=&M_k^{\alpha\left(\frac{2}{\alpha}+\frac{3}{\beta}-1\right)}\int_{-1}^0\|\nabla\zeta\|_{L^\beta(\Omega_{\sigma_\varepsilon}(s))}^\alpha ds\leq\varepsilon,\quad\mbox{ for large }k.
\end{align*}
Recalling that $\omega_k\geq\frac{\delta}{2}$ on $K$, it follows from (\ref{6.3}) that $\nabla\zeta_k\rightarrow\nabla\overline\zeta$ uniformly on $K$, and consequently, one can apply dominate convergence theorem to conclude
$$
\int_{-L}^0\|\nabla\zeta_k\|_{L^\beta(K_t)}^\alpha dx\rightarrow\int_{-L}^0\|\nabla\overline\zeta\|_{L^\beta(K_t)}^\alpha dt
$$
for any positive $L$, which, combined with the previous inequality, implies
$$
\int_{-L}^0\|\nabla\overline\zeta\|_{L^\beta(K_t)}^\alpha dt\leq\varepsilon,\quad\forall\varepsilon>0, L>0.
$$
Hence, $\nabla\overline\zeta\equiv0$ on $K$. Since $K$ is an arbitrary compact set contained in $\Omega$, one can easily infer that $\nabla\overline\zeta\equiv0$ on $\Omega$, which forces $\overline\zeta(x,t)\equiv\overline{\zeta_0}(t)$. The proof is complete.~~~~~~~~~~~~~~~~~~~~~~~~~~~~~~~~~~~~~~ $\square$

\par

\end{document}